%% file: saturated_main.tex
\documentclass[a4page,12pt]{article}
\input{macros}
\usepackage{float}
\makeatletter
\renewcommand\@makefntext[1]{%
  \parindent0.7em%
  \@thefnmark.\ #1}
\makeatother
\usepackage{etoolbox}
\makeatletter
\patchcmd{\maketitle}{\hb@xt@1.8em}{\hb@xt@0.7em}{}{}

\makeatother

\providecommand{\keywords}[1]
{
  \small	
  \textbf{\textit{\quad Keywords---}} #1
}

\newcommand{\appendixnumberline}[1]{Appendix\space}

\let\oldappendix\appendix
\makeatletter
\renewcommand{\appendix}{%
  \addtocontents{toc}{\let\protect\numberline\protect\appendixnumberline}%
  \renewcommand{\@seccntformat}[1]{Appendix~\csname the##1\endcsname:\ }%
  \oldappendix
}
\makeatother

\makeatletter
\let\@fnsymbol\@arabic
\makeatother

\title{On saturated triangulation-free convex geometric graphs}

\author{David Garber\thanks{\ 
			School of Mathematical Sciences, Faculty of Science,
			Holon Institute of Technology, Israel.
			Research partially supported by a grant from the Ariel University - Holon Institute of Technology Research foundation. 
\texttt{garber@hit.ac.il}}
\and Chaya Keller\thanks{\ 
			School of Computer Science,
			Ariel University, Israel.
			Research partially supported by Grant 1065/20 from the Israel Science Foundation, by Grant 2022792 from the Binational US-Israel Science Foundation, and by a grant from the Ariel University - Holon Institute of Technology Research foundation. 
\texttt{chayak@ariel.ac.il}}
		\and 
        Olga Nissenbaum\thanks{\ 
			School of Computer Science,
			Ariel University, Israel.
			Research partially supported by Grant 1065/20 from the Israel Science Foundation, by Grant 2022792 from the Binational US-Israel Science Foundation, and by a grant from the Ariel University - Holon Institute of Technology Research foundation. 
\texttt{olganissenbaum@gmail.com}}
        \and Shimon Aviram\thanks{\ 
	        School of Computer Sciences, Faculty of Science,
			Holon Institute of Technology, Israel.
			Research partially supported by a grant from the Ariel University - Holon Institute of Technology Research foundation. \texttt{shimonav@hit.ac.il}}
}

\date{}

\newtheorem{remark}[theorem]{Remark}

\begin{document}

\maketitle

\begin{abstract}
A {\it convex geometric graph} is a graph whose vertices are the corners of a convex polygon $P$ in the plane and whose edges are boundary edges and diagonals of the polygon. It is called {\it triangulation-free} if its non-boundary edges do not contain the set of diagonals of some triangulation of $P$. Aichholzer et al.~(2010) showed that the maximum number of edges in a triangulation-free convex geometric graph on $n$ vertices is ${{n}\choose{2}}-(n-2)$, and subsequently, Keller and Stein (2020) and (independently) Ali et al.~(2022) characterized the triangulation-free graphs with this maximum number of edges.

We initiate the study of the saturation version of the problem, namely, characterizing the triangulation-free convex geometric graphs which are not of the maximum possible size, but yet the addition of any edge to them results in containing a triangulation.

We show that, surprisingly, there exist saturated graphs with only $g(n)=O(n\log n)$ edges. Furthermore, we prove that for any $n>n_0$ and any $g(n)\leq t \leq {{n}\choose{2}}-(n-2)$, there exists a saturated graph with $n$ vertices and $t$ edges. In addition, we  
obtain a complete characterization of all saturated graphs whose number of edges is ${{n}\choose{2}}-(n-1)$, which is $1$ less than the maximum.
\end{abstract}

    

\section{Introduction}

\textbf{Tur\'{a}n-type problems in geometric settings.} Tur\'{a}n-type problems in graph theory study the maximum number $\ex(n,H)$ of edges in a graph $G$ on $n$ vertices that does not contain a copy of a `forbidden' configuration $H$ (e.g., a triangle) or, more generally, a member of some family $\mathcal{F}$ of `forbidden' configurations (e.g., a spanning tree of the vertex set), see, e.g., the surveys~\cite{Furedi91,Sidorenko95}. Initiated by Tur\'{a}n in 1941~\cite{Turan}, this research direction has become a central branch of extremal graph theory, with numerous works studying $\ex(n,H)$ for various configurations $H$.  One of the sub-directions studies Tur\'{a}n-type problems in a geometric setting, where the vertices of $G$ are points in the plane (sometimes, in convex position) and the edges of $H$ are realized by non-crossing segments. Well-studied examples include cases where $H$ is a set of $k$ pairwise disjoint edges~\cite{K79,KP96}, a non-crossing perfect matching~\cite{KP12}, a non-crossing spanning tree~\cite{Hernando}, a triangulation~\cite{AichholzerCMFHHHW10}, etc. 

Once $\ex(n,H)$ is determined, the next natural step is to characterize the family $\mathrm{Ex}(n,H)$ of extremal $H$-free graphs -- i.e., the graphs with $n$ vertices and $\ex(n,H)$ edges that do not contain a copy of $H$. Such  characterizations gave rise to interesting classes of examples, including caterpillar graphs (see~\cite{Caterpillar,KP12}), combs (see~\cite{KPRU13}) and semi-simple perfect matchings (see~\cite{KP13}), and had applications to the structure of the \emph{flip graphs} of the respective structures (see~\cite{Hernando,HHMMN99}).

\medskip \noindent \textbf{Saturation problems.} The subsequent natural step is to study \emph{saturated} $H$-free graphs, which are $H$-free graphs that are maximal \emph{under inclusion}, meaning that the addition of any edge to the graph would lead to containment of a copy of $H$. This direction was initiated in 1964 by Erd\H{o}s, Hajnal and Moon~\cite{EHM64}, who suggested to study the \emph{saturation number} $\sat(n,H)$ -- the minimum number of edges in a saturated $H$-free graph on $n$ vertices. The saturation problem was studied in dozens of papers, see the survey~\cite{CFFS21}. A further step is to study the \emph{saturation spectrum} -- i.e., the set $S(n,H)$ of possible values $\sat(n,H) \leq m \leq \ex(n,H)$ such that there exists a saturated $H$-free graph with $m$ edges. This direction was initiated in 1995 by Barefoot et al.~\cite{BarefootCFFH95} and was studied for various forbidden configurations $H$, including complete graphs~\cite{AminFGS13}, stars and paths~\cite{BalisterD18,FaudreeFGJT17}, odd cycles~\cite{GouldKK24}, etc. In a geometric setting, such a problem was studied by Kyncl et al.~\cite{KynclPRT15}. A natural question in this context, studied with respect to complete graphs~\cite{AminFGS13} and to paths~\cite{BalisterD18}, is characterizing the saturated $H$-free graphs whose number of edges is close to the extremal value $\ex(n,H)$. In addition to the intrinsic interest in this question, such characterizations were shown to be helpful in determining the saturation spectrum (see~\cite{BalisterD18}).    

\medskip \noindent \textbf{From extremal graphs to blockers.} In settings where the extremal size $\ex(n,H)$ is close to ${{n}\choose{2}}$, it is more convenient to pass to the complement graph and to consider \emph{blockers} -- i.e., sets of edges which intersect each copy of $H$. Clearly, the minimal size of a blocker is ${{n}\choose{2}}-\ex(n,H)$. A blocker is called \emph{saturated} if the deletion of any of its edges will transform it into a non-blocker. A blocker is saturated if and only if its complement is a saturated $H$-free graph, and thus, the maximum size of a saturated blocker is ${{n}\choose{2}}-\sat(n,H)$. The problem of characterizing saturated blockers is thus equivalent to the problem of characterizing saturated $H$-free graphs. In the geometric setting, blockers were characterized in the cases where $H$ is a non-crossing perfect matching~\cite{KP12}, a non-crossing spanning tree~\cite{Hernando}, a non-crossing  Hamiltonian path~\cite{KP16}, etc.

\medskip
\noindent \textbf{Blockers for triangulations in convex geometric graphs.} One of the geometric settings in which blockers were studied is \emph{blockers for triangulations in a convex geometric graph} -- namely, the case where $G$ is a graph whose vertices are the $n$ vertices of a convex polygon $P$ in the plane and the forbidden family $\mathcal{T}$ is the family of sets of diagonals which, together with all the boundary edges, form a triangulation of $P$. Triangulations of convex polygons were studied from various points of view, such as finding the optimal triangulation with respect to different criteria (see ~\cite{KT06,Kli80}) and studying the flip graph $\mathcal{T(G)}$ of triangulations (see~\cite{HN99,Lee89}), whose properties are related to deep results in hyperbolic geometry, as shown in the seminal paper of Sleator, Tarjan and Thurston~\cite{STT88}. For more information on triangulations of a convex polygon, see \cite{DRS10}.

The basic problem of studying the minimum size $\ex_B(n,
\mathcal{T})$ of a blocker for triangulations in a convex geometric graph was introduced by Aichholzer et al.~\cite{AichholzerCMFHHHW10} who showed that $\ex_B(n,\mathcal{T})=n-2$. Keller and Stein \cite{KellerS20} obtained the following complete characterization of the family $\mathrm{Ex}_B(n,\mathcal{T})$ of minimum-sized blockers:
\begin{theorem}[\cite{KellerS20}]\label{Thm:MinBlockerIntro}
Any blocker $B$ of size $n-2$ for triangulations of a convex polygon $P=\langle 0,1,\ldots,n-1 \rangle$ is 
(up to a cyclical rotation of $P$) of the form $B = B_1\cup B_2$, where
\[
B_1 = \{(0, 2),(1, 3),(2, 4), \dots ,(m, m + 2)\}, \mbox{ and }
\]
\[
B_2 = \{(m+3,i_1),(m+4,i_2),
(m+5,i_3), \dots ,(n-1, i_{n-3-m})\}, 
\]
for some $1\leq m \leq n-3$ and $1\leq i_j \leq m+1$, such that if $|i_j-i_k|\geq 2$, then the diagonals $(m+j+2, i_j )$ and $(m + k + 2, i_k)$ do not cross.  
\end{theorem}
In words, such a blocker $B$ consists of two sets of edges:
\begin{itemize}
\item The first is the {\em boundary net} of the blocker, which is a sequence of $m+1$ consecutive diagonals of the form $(i,i+2)$ (called \emph{ear-covers}), which cover the path $\langle 0,1,\ldots,m+2 \rangle$ on the boundary of the polygon.
\item The second is a set of $n-3-m$ leaf edges (called \emph{beams}) that connect each of the vertices $m+3,m+4,\ldots,n-1$ to an internal vertex of the path $\langle 0,1,\ldots,m+2 \rangle$, such that two beams whose endpoints on the path are not consecutive do not cross each other.
\end{itemize}
\cref{Fig:ExBlockers1} illustrates some examples for such blockers, where the corresponding boundary net is drawn by bold lines. 

\medskip

\begin{figure}[h]
    \centering
    \includegraphics[width=\linewidth]{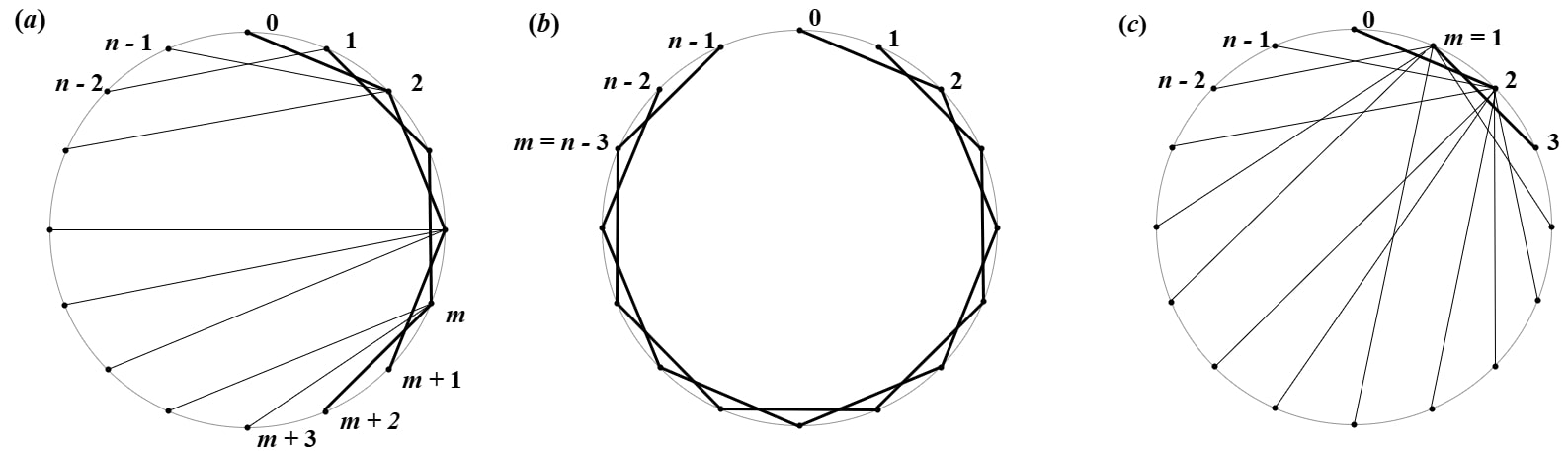}
    \caption{Examples of minimum-sized blockers, where the boundary net is drawn by bold lines. Figure (b) demonstrates a minimum-sized blocker with no beams. Figure (c) demonstrates a minimum-sized blocker with the shortest possible boundary net.} 
 \label{Fig:ExBlockers1}
\end{figure}

\begin{remark}
The same characterization was obtained independently by Ali et al.~\cite{AliCTK22}, who also studied blockers of sizes $n-1$ and $n$. In the characterization of blockers of size $n-1$ in their paper, the saturation condition was not considered. The blockers are divided into two types: the first type contains only blockers which are not saturated, and the second type contains both saturated and non-saturated blockers (see~\cite[Definition 3 and Example 2]{AliCTK22}). 
\end{remark}

\medskip \noindent \textbf{Our results.} In this paper, we initiate the study of the saturation version of this problem -- namely, studying the \emph{saturated blockers} which are not of a minimum size, but are still minimal with respect to inclusion. 

Our first main result concerns the \emph{saturation spectrum} -- namely, the set of values $t$ such that there exists a saturated blocker of size $t$. 
We prove that surprisingly there exist saturated blockers with as many as $g(n)=\binom{n}{2}-O(n\log n)$ edges. This means that the corresponding graph has only $O(n\log n)$ edges, and yet, the addition of any edge to it will make it contain a triangulation. Furthermore, we show that all values between the minimum possible number $n-2$ and $g(n)$ belong to the saturation spectrum, for a sufficiently large $n$. Formally, we show the following:
\begin{theorem}\label{thm:saturation_spectrum}
There exists a universal constant $C>0$ and an integer $n_0$ such that for any integer $n>n_0$ and any $n-2 \leq t \leq \frac{n^2}{2}-C n \log n$, there exists 
a blocker for triangulations in an $n$-gon which is saturated (i.e., minimal with respect to inclusion) and consists of exactly $t$ edges.
\end{theorem}

The proof makes use of several tailor-made families of saturated blockers, obtained by a combination of geometric constructions together with combinatorial fractal structures.

\medskip

Our second main result is a complete characterization of the saturated blockers of size $n-1$, in the spirit of the results of~\cite{AminFGS13,BalisterD18}, which characterized the almost-extremal saturated graphs with respect to forbidden complete graphs and paths, respectively. 

Our characterization involves three mini-subgraphs, which we call \emph{seagull, butterfly}, and \emph{bouquet}, presented in \cref{Fig:Bouquet1}. 

\medskip

\begin{figure}[h]
    \centering
    \includegraphics[width=\linewidth]{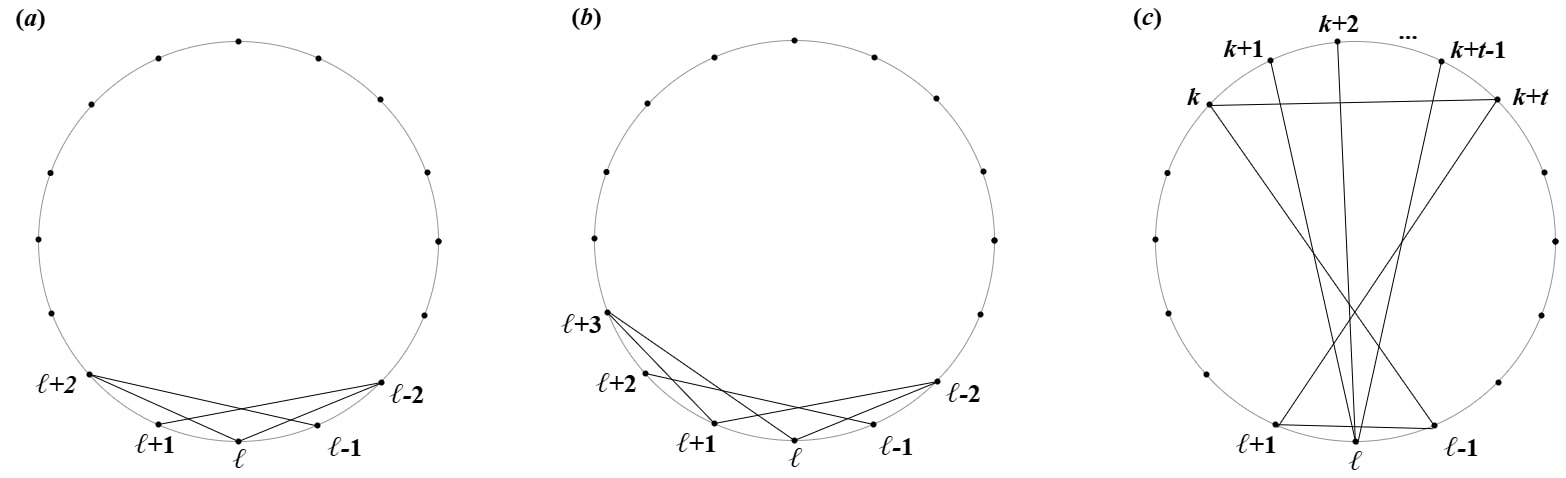}
    \caption{(a) Example of a \emph{seagull} subgraph; (b) Example of a \emph{butterfly} subgraph; (c) Example of a \emph{bouquet} subgraph.} \label{Fig:Bouquet1}
\end{figure}

The complete characterization of the saturated blockers of size $n-1$ is the following:
\begin{theorem}\label{Thm:Main-Intro}
For $n>5$, any blocker for triangulations of a convex $n$-gon having $n-1$ edges
is of one of the following three types (see \cref{Fig:Thm2_Blockers}):
\addtolength{\leftmargini}{2cm}
\begin{itemize}
    \item[\bf{Type 1:}] $B=\Sea\cup B_1\cup B_2$, where for some $2<\ell<m$,
    $$\Sea=\{(\ell-2,\ell),(\ell-2,\ell+1),(\ell-1,\ell+2),(\ell,\ell+2)\}$$  is a ``seagull'' subgraph, $$B_1=\{(0,2),\dots,(\ell-3,\ell-1),(\ell+1,\ell+3),\dots,(m,m+2)\}$$ is an incomplete boundary net, and $B_2 = \{ (m+3,i_1), (m+4,i_2),\dots,(n-1,i_{n-3-m})\}$  is a set of beams for some $4\leq m\leq n-3$, $i_j\in\{1,\dots,\ell-1,\ell+1,\dots,m+1\}$ such that if $|i_{j_1}-i_{j_2}|\geq 2$, then $(m+2+j_1,i_{j_1})$ and $(m+2+j_2,i_{j_2})$ do not cross.

\medskip

    \item[\bf{Type 2:}] $B=\Buf\cup B_1\cup B_2$, where for some $2<\ell<m-1$, $$\Buf=\{(\ell-2,\ell),(\ell-2,\ell+1),(\ell-1,\ell+2),(\ell,\ell+3),(\ell+1,\ell+3)\}$$ is a ``butterfly'' subgraph, $$B_1=\{(0,2),\dots,(\ell-3,\ell-1),(\ell+2,\ell+4),\dots,(m,m+2)\}$$ is an incomplete boundary net, and $B_2 = \{ (m+3,i_1), (m+4,i_2),\dots,(n-1,i_{n-3-m})\}$  is a set of beams for some $5\leq m\leq n-3$, $i_j\in\{1,\dots,\ell-1,\ell+2,\dots,m+1\}$ such that if $|i_{j_1}-i_{j_2}|\geq 2$, then $(m+2+j_1,i_{j_1})$ and $(m+2+j_2,i_{j_2})$ do not cross.

\medskip

    \item[\bf{Type 3:}] $B=\Bouq\cup B_1\cup B_2$, where $$\Bouq=\{(\ell-1,\ell+1),(k,\ell-1),(k+t,\ell+1),(k,k+t), (k+1,\ell),\dots,(k+t-1,\ell)\}$$ is a ``bouquet'' subgraph, $B_1=B_{\rm 1,L}\cup B_{\rm 1,R}$ is an incomplete boundary net,  and $B_2 = B_{\rm 2,L}\cup B_{\rm 2,R}$ is  a set of beams, where:
    \begin{itemize}
    \item[$\bullet$] $0<\ell<m+2< k< k+1< k+t\leq n-1$;
    \item[$\bullet$] $B_{\rm 1,L}=\left\{ \begin{array}{ll}\{(\ell,\ell+2),\dots,(m,m+2)\},& \ell<m+1\\ \emptyset, & \ell=m+1\end{array}\right.$; \vspace{5pt}\\
    $B_{\rm 1,R}=\left\{\begin{array}{ll}\{(0,2),\dots,(\ell-2,\ell)\}, & \ell>1 \\ \emptyset, & \ell=1 \end{array}\right.$;
    
    \item[$\bullet$]  $B_{\rm 2,L}=\left\{ \begin{array}{ll}\{ (m+3,i_1), \dots,(k-1,i_{k-m-3})\}, & k>m+3 \\\emptyset, & k=m+3\end{array}\right.$;\vspace{5pt}\\ 
    $B_{\rm 2,R}=\left\{\begin{array}{ll}\{(k+t+1,i_{k+t-m-1}),\dots,(n-1,i_{n-3-m})\}, & k+t<n-1 \\\emptyset, & k+t=n-1\end{array}\right.$;
    
    \item[$\bullet$] $i_1,\dots,i_{k-m-3}\in\{\ell,\dots,m+1\}$, $i_{k+t-m-1},\dots,i_{n-3-m}\in\{1,\dots,\ell\}$;
    
    \item[$\bullet$] if $|i_{j_1}-i_{j_2}|\geq 2$, then $(m+2+j_1,i_{j_1})$ and $(m+2+j_2,i_{j_2})$ do not cross.
    \end{itemize}
\end{itemize}
\end{theorem}
In words, each saturated blocker for triangulations of size $n-1$ looks like a modification of a minimum-sized blocker having $n-2$ edges, with some subgraph of $k$ edges replaced by a ``seagull'', a ``butterfly'' or a ``bouquet'' of size $k+1$, see the different parts of \cref{Fig:Thm2_Blockers}. 

\begin{figure}[h]
    \centering
    \includegraphics[width=\linewidth]{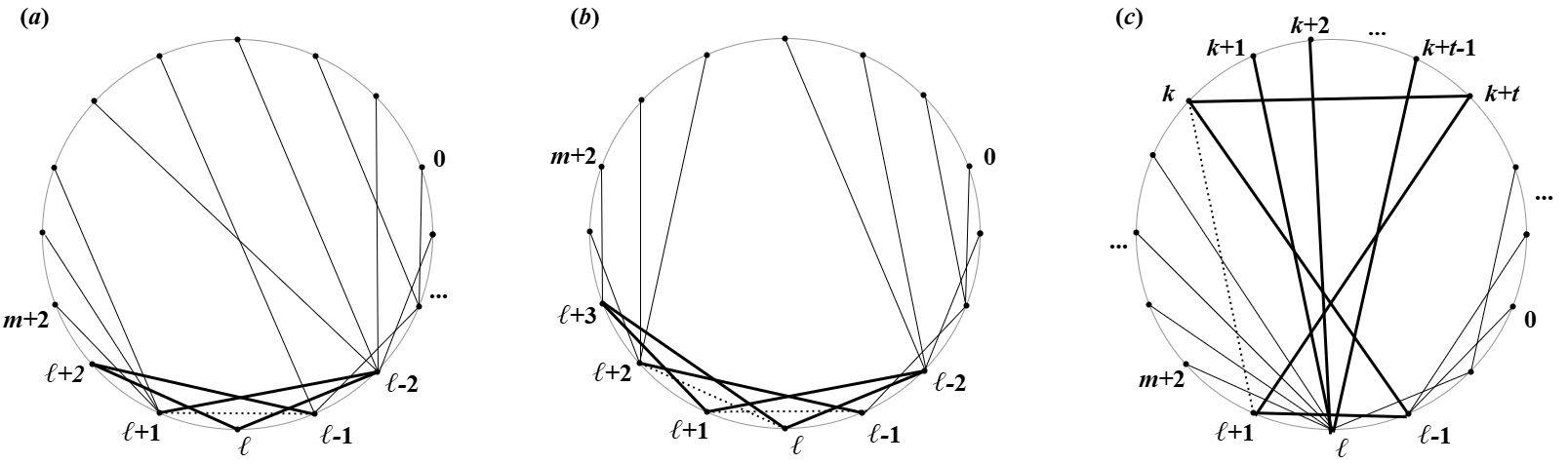}
    \caption{Examples of blockers for triangulations of a convex $n$-gon having $n-1$ edges: Part (a) demonstrates a blocker of Type 1. Part (b) demonstrates a blocker of Type 2. Part (c) demonstrates a blocker of Type 3 with $|B_{\rm 1,L}|=1$, $|B_{\rm 1,R}|=2$, $|B_{\rm 2,L}|=3$, $|B_{\rm 2,R}|=2$. The special subgraphs are drawn by bold lines, and the removed edges (as described after \cref{cor:deleted_edges}) are drawn by dotted lines.  
    } \label{Fig:Thm2_Blockers}
\end{figure}

Our characterization implies a \emph{stability result}, which seems to be of independent interest:
\begin{corollary}\label{cor:deleted_edges}
    Let $B$ be a saturated blocker for triangulations of a convex polygon with $|E(B)|=n-1$. Then there exists a (minimum-sized) blocker $B'$ having $n-2$ edges, such that the symmetric difference of $E(B),E(B')$ satisfies $|E(B) \triangle E(B')| \leq 5$.
\end{corollary}
Indeed, it can be seen from the characterization of Theorem~\ref{Thm:Main-Intro} that each saturated blocker $B$ having $n-1$ edges is obtained from some minimum-sized blocker $B'$ by one of the following two constructions:
\begin{enumerate}
    \item[(1)] Removing one edge of $B'$ (which is, in the notations of \cref{Fig:Bouquet1}(a,c), $(\ell-1,\ell+1)$ in the case of a seagull, and some beam from the vertex $k$ in the case of a bouquet), and adding two new edges (which are $(\ell-1,\ell+2),(\ell+1,\ell-2)$ in the case of a seagull, and $(k,\ell-1),(k,k+t)$ in the case of a bouquet).

    \item[(2)] Removing two consecutive ear-covers of $B'$ (which are, in the notations of \cref{Fig:Bouquet1}(b), $(\ell,\ell+2),\break(\ell-1,\ell+1) $ in the case of a butterfly), and adding the edges $(\ell,\ell+3),(\ell+2,\ell-1),(\ell+1,\ell-2)$.
\end{enumerate}

\medskip

The paper is organized as follows.  
Section \ref{Sec:Prelim} presents definitions, notations and simple observations that will be used throughout the paper.
In Section \ref{sec:spectrum}, we prove that almost all the range between $n-2$ and $\binom{n}{2}$ is contained in the saturation spectrum of blockers for triangulations. 
In Section \ref{Sec:Full_Char}, we prove \cref{Thm:Main-Intro} which provides a complete characterization of the $(n,n-1)$-blockers for $n>6$. The $(n,n-1)$-blockers for $n \leq 6$ are characterized in \cref{app:small_cases}.


\section{Preliminaries and notations}
\label{Sec:Prelim}
In this section, we present definitions, notations and simple observations that will be used throughout the paper.

\medskip \noindent \textbf{Standard definitions.} Let $G$ be the complete geometric graph on a set $P$ of $n$ vertices, realized in the plane as the vertex set of a convex polygon $C$. 
The \emph{degree} of a vertex $v\in V(G)$, denoted by $\deg(v)$, is the number of edges that emanate from $v$. 
The degree of $v$ with respect to a subgraph $B$ of $G$ is denoted by $\deg_B(v)$. We say that two edges $e_1, e_2 \in E(G)$ {\em cross} or {\em intersect,} if they share an interior point.
The {\em order} of an edge $e=(i,j)\in E(G)$, denoted by $o(e)$, is $o(e)=\min\{|i-j|,n-|i-j|\}$.
Note that the edges of the $n$-gon $C$ are all of order 1. A {\em diagonal} of $C$ is an edge of $G$ of order $\geq 2$. 
A diagonal of $C$ of order $2$ is called an {\em ear-cover}. We say that the ear-cover $(i-1,i+1)$ {\em covers} the vertex $i$.

A \emph{triangulation} of a convex polygon $C$ is a maximal-for-inclusion pairwise non-crossing set of diagonals of $C$.  
Let $\TriF$ denote the family of triangulations of $C$. Note that
any triangulation of an $n$-gon contains $n-3$ diagonals.

\medskip \noindent \textbf{$(n,k)$-Blockers.}
As was described in the introduction, a set $B$ of edges in $G$ is a {\em blocker} for $\TriF$ (or simply, a blocker) if it contains an edge in common with every element of $\TriF$.
We define {\em $(n,k)$-blockers} as the \emph{saturated} blockers of size $k$ for triangulations of a convex $n$-gon $C$. This means that no edge can be removed from an $(n,k)$-blocker such that the remaining set will be a blocker.
Note that the $(n,n-2)$-blockers are the minimum-sized blockers for triangulations of a convex $n$-gon, as was proved in \cite{AichholzerCMFHHHW10}. 

\medskip \noindent \textbf{Vertex deletion.}
We denote by $C\setminus\{i\}$ the polygon obtained from $C$ by deleting the vertex $i$ and adding the
edge $(i-1, i+1)$, and by $B\setminus\{i\}$ the restriction of a blocker $B$ to a blocker of $C\setminus\{i\}$ (i.e,
$B\setminus\{i\}=B\setminus\{(i,j)\mid (i,j)\in E(B),\ 0\leq j\leq n-1\} {\setminus \{(i-1,i+1)\}}$ ).

\medskip \noindent \textbf{Conflicting edges.}
In the graph containing a boundary net $B_1$ as in \cref{Thm:MinBlockerIntro}, we call the beams $(m+2+j,i_{j})$ and $(m+2+k,i_k)$ {\em conflicting}, if $j<k$, and $i_k\geq i_j+ 2$ (see \cref{Fig:Obs13_Def14}(a)).

\begin{figure}[h]
    \centering
    \includegraphics[width=0.6\linewidth]{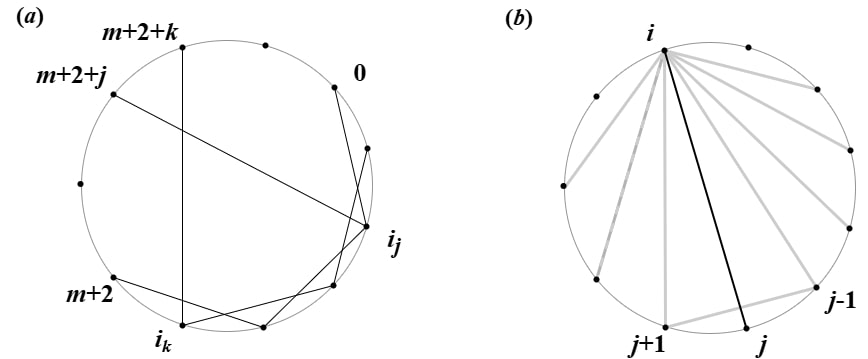}
    \caption{Part (a) demonstrates an example of conflicting edges. Part (b) demonstrates a triangulation $T$ (drawn by gray lines) from \cref{Obs:Covered} for the graph with the edge $(i,j)$, where ${\rm deg}(i)=1$, and $j$ is uncovered.
    } \label{Fig:Obs13_Def14}
\end{figure}

\cref{Thm:MinBlockerIntro} above supplies a complete characterization of all $(n,n-2)$-blockers, which are also the minimum-sized blockers for triangulations.

According to that theorem, in a minimum-sized blocker $B$, there are no conflicting beams. In an $(n,n-1)$-blocker, however, conflicting beams are possible, see e.g. Figure \ref{Fig:Thm2_Blockers}(c).

\medskip 

We shall use the following simple observations. The first two of them were proved in~\cite{AichholzerCMFHHHW10}.

\begin{observation}[\cite{AichholzerCMFHHHW10}]\label{Obs:Isolated}
    No blocker for $\TriF$ contains an isolated vertex.
\end{observation}

\begin{observation}[\cite{AichholzerCMFHHHW10}]\label{Obs:MinB_n-2}
    The size of each blocker for the family of triangulations $\TriF$ of a convex $n$-gon $C$ is at least $n-2$.
\end{observation}

\begin{observation}\label{Obs:Covered}
    If the vertex $i$ is connected only to $j$ in some $(n,k)$-blocker $B$, then $B$ contains the ear-cover $(j-1,j+1)$.
\end{observation}
\begin{proof}
Otherwise, $B$ misses the triangulation $T=\{(j-1,j+1), (i,j+1),(i,j+2),\dots,(i,j-1)\}$, see \cref{Fig:Obs13_Def14}(b).
\end{proof}

\section{The saturation spectrum}\label{sec:spectrum}

In this section, we prove  Theorem \ref{thm:saturation_spectrum} above, that almost all the range of values between $n-2$ and $\binom{n}{2}$ is contained in the saturation spectrum of blockers for triangulations. We recall its statement: 

\medskip

\noindent \textbf{Theorem 
\ref{thm:saturation_spectrum}.}
    There exists a universal constant $C>0$ and an integer $n_0$ such that for any integer $n>n_0$ and $n-2 \leq t \leq \frac{n^2}{2}-C n \log n$, there exists an $(n,t)$-blocker.

\medskip
 
Clearly, the lower bound is tight by \cref{Obs:MinB_n-2}. The upper bound is only $(C n \log n)$-far from the total number of edges in the complete graph on $n$ vertices.

We prove this theorem constructively, by building an $(n,t)$-blocker for $t$ incrementally growing from $t=n-2$ to $t=\binom{n}{2}-Cn\log n$. 

\subsection{The building blocks}

We first present the two building blocks which we repeatedly use in our constructions. The first building block is a specific type of a saturated blocker, which we call a \emph{quadrilateral} blocker.
\begin{proposition}\label{Obs:Quadrilateral}
In a convex $n$-gon, we divide the set of vertices into four consecutive non-empty subsets: $V=V_R\cup V_B\cup V_L\cup V_T$, such that for some $0<a<b<c<n-1$, $V_R=\{0,\dots,a-1\}$, $V_B=\{a,\dots,b-1\}$, $V_L=\{b,\dots,c-1\}$, and $V_T=\{c,\dots,n-1\}$. Then the graph $B_Q$ whose edge set consists of all possible edges $(i,j)$ where $i\in V_R$, $j\in V_L$, or where $i\in V_T$, $j\in V_B$, is an $(n,m)$-blocker for $m=|V_R|\cdot|V_L|+|V_T|\cdot|V_B|$. 
\end{proposition}
An example of an $(n,m)$-blocker $B_Q$ which is, in fact, a disjoint union of two complete bipartite  graphs, is given in \cref{Fig:Quadrilateral}. 
\begin{figure}[h]
    \centering
    \includegraphics[scale=0.3]{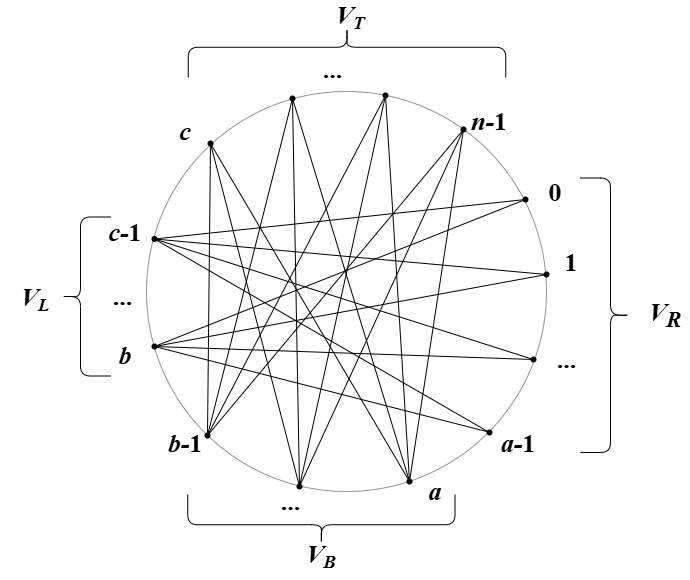}
    \caption{An example of a quadrilateral blocker $B_Q$ considered in \cref{Obs:Quadrilateral}.} 
 \label{Fig:Quadrilateral}
\end{figure}
\begin{proof}[Proof of \cref{Obs:Quadrilateral}]
First, we prove that $B_Q$ is indeed a blocker.
Assume to the contrary that there exists a triangulation $T_0$ avoiding $B_Q$, hence it must contain edges that connect two neighboring subsets, without loss of generality, $V_B$ and $V_R$. Let us denote the leftmost such edge as $e$, and consider the triangle in the triangulation that contains $e$ and is to the left of $e$. As $e$ is the leftmost of all the edges that connect $V_B$ and $V_R$, the third vertex of this triangle belongs to either $V_T$ (and then there is an edge connecting $V_T$ and $V_B$), or $V_L$ (and then there is an edge connecting $V_L$ and $V_R$), and hence $T_0 \cap B_Q\neq \emptyset$, a contradiction. 

Second, we prove that the blocker $B_Q$ is saturated (i.e. minimal with respect to inclusion). Let $(i,j)\in B_Q$ and assume without loss of generality, that $i\in V_T$, $j\in V_B$. We show that the set of edges $B_Q\setminus\{(i,j)\}$ is not a blocker, by explicitly constructing a triangulation it misses. We choose arbitrarily two vertices $k\in V_L$ and $\ell\in V_R$, and construct a triangulation by taking the edge $(i,j)$ and adding to it all possible edges emanating from $k$ to the left of $(i,j)$, and all possible edges emanating from $\ell$ to the right of $(i,j)$ (see Figure~\ref{Fig:Quadril_Proof}). Clearly, this triangulation is not blocked by $B_Q\setminus\{(i,j)\}$. This shows that $B_Q$ is saturated, completing the proof. 
\begin{figure}[h]
    \centering
    \includegraphics[scale=0.3]{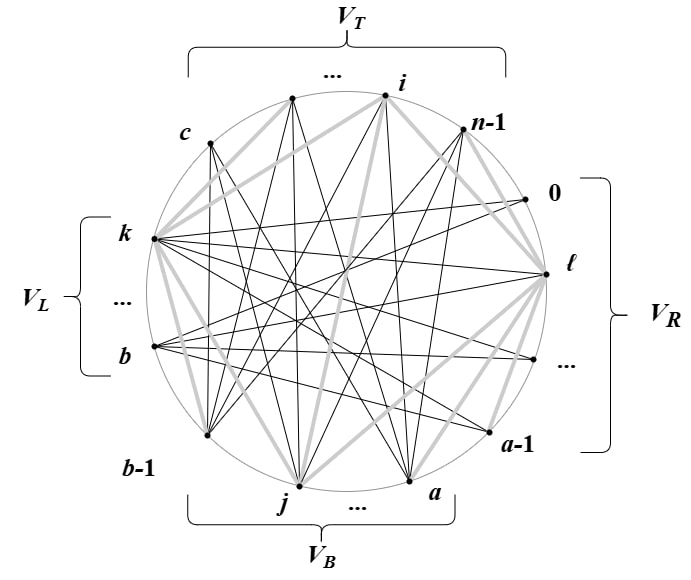}
    \caption{An example of a triangulation (drawn by gray lines) of  $B_Q\setminus\{(i,j)\}$ in the proof of \cref{Obs:Quadrilateral}.} 
 \label{Fig:Quadril_Proof}
\end{figure}
\end{proof}

The second building block allows us to construct larger saturated blockers from given smaller ones:

\begin{proposition}\label{Obs:Matrioshka}
Given a blocker $B_Q$ as in \cref{Obs:Quadrilateral} with $|V_L|\geq 3$, $|V_R|\geq 3$, $|V_T|\geq 2$ and $|V_B|\geq 2$, the graph $B_M$ obtained from $B_Q$ by removing the edges $(0,c-1)$ and $(a-1,b)$ and adding two saturated blockers $B_T$ and $B_B$ for the convex polygons $V_T\cup\{0,c-1\}$, and $V_B\cup\{a-1,b\}$, respectively, is a saturated blocker of size $|V_T||V_B|+|V_L||V_R|+|B_T|+|B_B|-2$ (see an example for a blocker $B_M$ in \cref{Fig:Matrioshka}).

\end{proposition}

\begin{figure}[h!]
    \centering
    \includegraphics[scale=0.33]{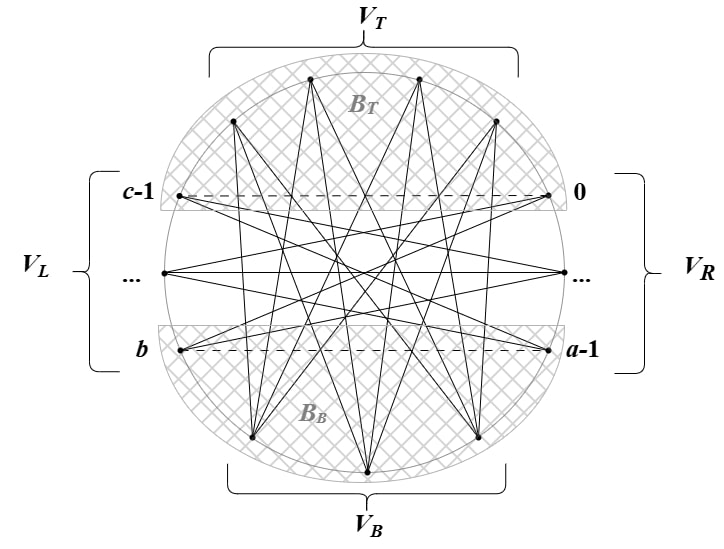}
    \caption{An example of a saturated blocker $B_M$ considered in \cref{Obs:Matrioshka}. The blockers $B_T$ and $B_B$ are drawn schematically in gray. The removed edges are drawn as dashed lines.} 
 \label{Fig:Matrioshka}
\end{figure}

\begin{proof}
Again, we start with the proof that $B_M$ is a blocker. Note that any triangulation which does not contain any of the edges $(0,c-1)$ and $(a-1,b)$ is blocked by $B_Q$. Any triangulation which contains at least one of the edges $(0,c-1)$, $(a-1,b)$ is blocked by one of the additional blockers $B_T$ or $B_B$, respectively. Indeed, any triangulation containing the edge $(0,c-1)$ should be completed by a triangulation of the convex polygon $V_T\cup\{0,c-1\}$ which is blocked by $B_T$. The same holds for $(a-1,b)$ with respect to $B_B$.

\medskip

Next, we prove that $B_M$ is saturated. We split the proof into three cases, depending on the edge that is removed from $B_M$:

\begin{itemize}
\item[(1)] If the removed edge $e$ is one of the edges of $B_T$, then the set $B_M\setminus \{e\}$ misses a triangulation of $V_T\cup\{0,c-1\}$. This triangulation can be completed to a triangulation of the convex $n$-gon by adding the edges $(0,c-1)$, $(c-1,a)$, $(0,a)$, and all edges emanating from $c-1$ to the vertices of the set $\{a+1,\dots,c-3\}$, and from $0$ to the vertices of $\{2,\dots,a-1\}$ (see \cref{Fig:Matrioshka_Cases_1}(a)).
The case of removing an edge from $B_B$ is symmetric.

\begin{figure}[h!]
\centering
\includegraphics[scale=0.33]{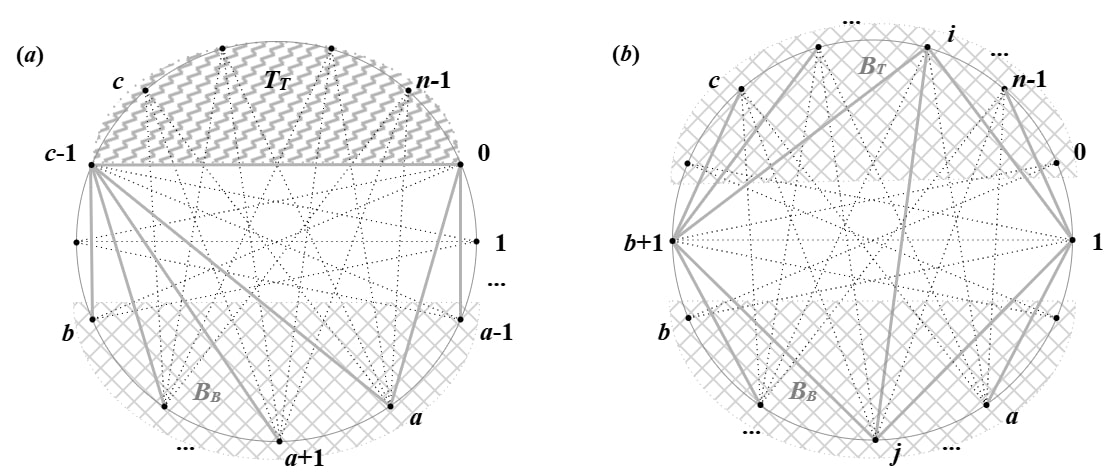}
\caption{Cases of a missing edge in $B_M$ in the proof of \cref{Obs:Matrioshka}: Part (a) illustrates the case when some edge from $B_T$ is missing; Part (b) illustrates the case when some edge connecting $V_T$ and $V_B$ is missing. \\
In both cases, the triangulations are drawn as gray bold lines, the edges of $B_M$ are drawn as dotted lines, $T_T$ is a triangulation of a part of the convex above the edge $(0,c-1)$.} 
\label{Fig:Matrioshka_Cases_1}
\end{figure}

\item[(2)] If an edge of the form $e=(i,j)$ where $i\in V_T$ and  $j\in V_B$ is removed, then the set $B_M\setminus \{e\}$ misses the triangulation which consists of $e$, along with all edges emanating from the vertex $1$ to the vertices to the right of $e$, and all edges emanating from the vertex $b+1$ to vertices to the left of $e$ (see \cref{Fig:Matrioshka_Cases_1}(b)).

\item[(3)] If an edge of the form $e=(k,\ell)$ where $k\in V_L$ and $\ell\in V_R$ is removed, then the set $B_M\setminus \{e\}$ misses the triangulation which consists of $e$, the two triangles $\triangle{k \ell c}$ and $\triangle{k \ell a}$, and all diagonals outside these triangles emanating from $k$ to the vertices placed to the left of $(c,a)$, and from $\ell$ to the vertices placed to the right of $(c,a)$. This triangulation can take one of three possible forms, depending on which of the ends of the edge $e$ belong to either $V_T \cup \{0,c-1\}$ or $V_B\cup \{a-1,b\}$,
see \cref{Fig:Matrioshka_Cases_2}.

\begin{figure}[h!]
\centering
\includegraphics[width=\linewidth]{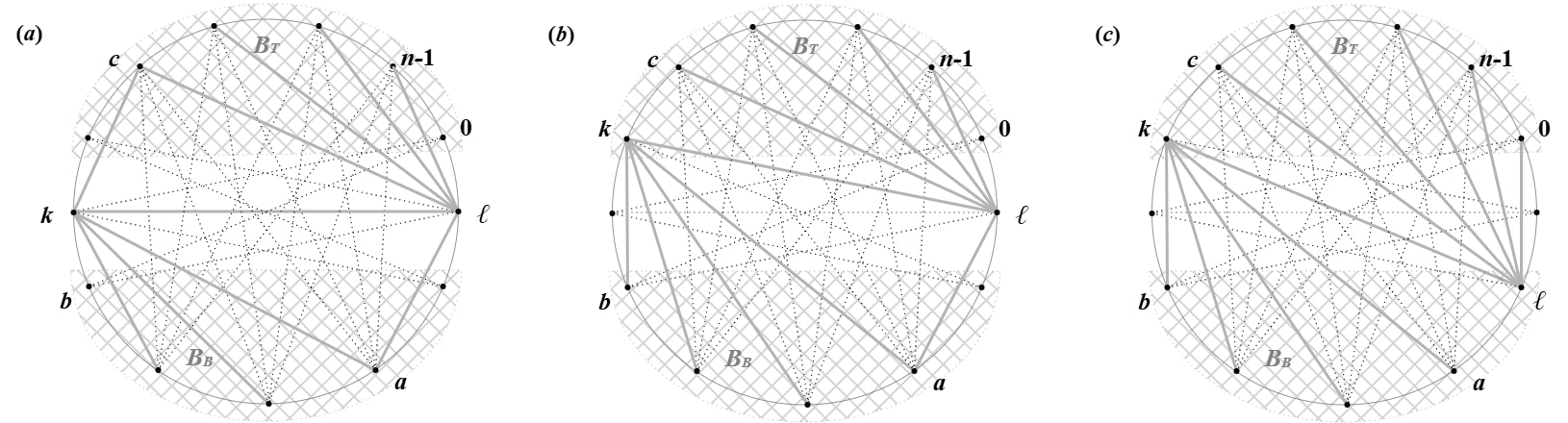}
\caption{The case of a missing edge $e=(k,\ell)$ connecting $V_L$ and $V_R$ in $B_M$ in the proof of \cref{Obs:Matrioshka}: Part (a) presents the case where both ends of $e$ do not belong to either $V_T \cup \{0,c-1\}$ or $V_B \cup \{a-1,b\}$, part (b) presents the case where exactly one of the ends belongs to either $V_T \cup \{0,c-1\}$ or $V_B\cup \{a-1,b\}$, and part (c) presents the case where both ends of $e$ are included in either $V_T \cup \{0,c-1\}$ or $V_B \cup \{a-1,b\}$.    The triangulations are drawn as gray bold lines and the edges of $B_M$ are drawn as dotted lines.} 
\label{Fig:Matrioshka_Cases_2}
\end{figure}

\end{itemize}

In this way, we see that none of the edges of the new blocker $B_M$ can be removed without missing a triangulation, and hence $B_M$ is saturated.
The statement concerning the number of edges in $B_M$ is proved by counting.
\end{proof}

\subsection{Proving the  saturation spectrum of blockers for triangulations}

In order to prove \cref{thm:saturation_spectrum}, we partition the range of the sizes of blockers into four intervals, and handle each of them using a different method. This is done in Lemmas \ref{Lem:First_Spectrum}--\ref{Lem:Third_Spectrum} below.

\subsubsection{The $(n,t)$-blockers for an integer $t$ in the range $\left[n-2, \frac{n^2}{8}+\frac{n}{4}-\frac{11}{8}\right]$}

This range is covered in the following lemma:
\begin{lemma}\label{Lem:First_Spectrum}
For any $n\geq 5$ and all integers $t$ satisfying $n-2\leq t\leq \frac{n^2}{8}+\frac{n}{4}-\frac{11}{8}$, there exists an $(n,t)$-blocker.
\end{lemma}

\begin{proof}
We prove this lemma by constructing an $(n,t)$-blocker $B_Q$ as in \cref{Obs:Quadrilateral} for any given $t$ in the following manner. 
We fix some $k\geq 1$, and take $B_Q$ where $|V_T|=k$, $|V_R|=k+1$, and $|V_L|$ runs over the values from 1 to $n-2(k+1)$ (hence, $|V_B|$ runs down over the values from $n-2(k+1)$ to 1). As a result, $t$ grows from $$t_{\min}(k)=(k+1)\cdot 1+ k(n-2(k+1))=kn-k+1-2k^2$$ to $$t_{\max}(k)=(k+1)(n-2(k+1))+k \cdot 1=kn+n-3k-2-2k^2,$$ with the increment 1 (which follows from the fact that by removing one vertex from $|V_B|$, we decrease the total number of edges by $k$, and by adding one vertex to $|V_L|$, we increase the number of edges by $k+1$, see \cref{Fig:First_Spectrum}). 

\begin{figure}[h!]
    \centering
    \includegraphics[scale=0.29]{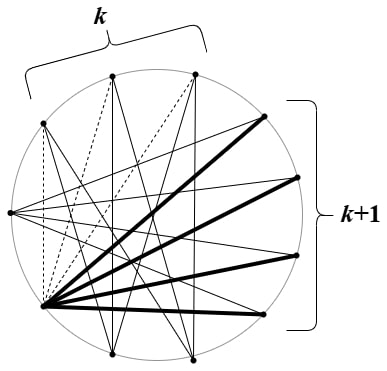}
    \caption{An example of increasing the value of $t$ by 1, for an $(n,t)$-blocker of the form $B_Q$ with a fixed $k$. Removed edges are drawn as dotted lines, added edges are drawn as bold lines.} 
 \label{Fig:First_Spectrum}
\end{figure}

Note that the minimal value of $n$ for which $B_Q$ with $|V_R|=|V_T|+1$ can be constructed is $n=5$. In this case, the only suitable value of $k$ is 1, and we have $|V_B|=|V_L|=1$.

For a fixed $k$, we have: 
$$t_{\max}(k)-t_{\min}(k)=(kn+n-3k-2-2k^2)-(kn-k+1-2k^2)=n-2k-3.$$ 
This difference is positive for all $k<\frac{n-3}{2}$. For any $k>0$, the intervals $\left[t_{\min}(k), t_{\max}(k)\right]$ and\break $\left[t_{\min}(k+1),t_{\max}(k+1)\right]$ overlap, since 
$$t_{\max}(k)-t_{\min}(k+1)=(kn+n-3k-2-2k^2)-((k+1)n-(k+1)+1-2(k+1)^2)=2k>0.$$ 

The minimal possible value of $t$ is $t_{\min}(1)=n-2$. The maximal size of a blocker of the form described above is $t_{\max}(k)$, for $k$ such that $\frac{dt_{\max}(k)}{dk}=n-3-4k=0$, which implies: $k=\frac{n-3} {4}$ (if this number is an integer). This maximal size is: $t_{\max}\left(\frac{n-3}{4}\right)=\frac{n^2}{8}+\frac{n}{4}-\frac{7}{8}$.
If $n\not\equiv 3(\text{mod }4)$, the maximum is obtained for $k=\left[\frac{n-3}{4}\right]$ (where $[\cdot ]$ is the floor function), and it is not less than $\frac{n^2}{8}+\frac{n}{4}-\frac{11}{8}$, which is the value for $n\equiv 1 (\text{mod }4)$.
\end{proof}

\subsubsection{The $(n,t)$-blockers for an integer $t$ in the range $\left[\frac{n^2}{8}+\frac{n}{2}-2,  \frac{n^2}{4}+\frac{n^2}{16}-3n+{12\frac{13}{16}}\right]$}

In order to cover the next interval in the saturation spectrum, we construct in the following lemma the set of blockers $B_M$ defined in \cref{Obs:Matrioshka}, where the sub-blockers $B_T$ and $B_B$ are both of the type $B_Q$ taken from the construction in \cref{Lem:First_Spectrum}. We emphasize that the interval of values of $t$ covered in this case does not overlap with the interval covered in the first case, and the gap between these two intervals will be covered in the next case.

\begin{lemma}\label{Lem:Second_Spectrum}
For any $n\geq 25$ and all integers $t$ satisfying $\frac{n^2}{8}+\frac{n}{2}-2\leq t\leq  \frac{n^2}{4}+\frac{n^2}{16}-3n+12\frac{13}{16}$, there exists an $(n,t)$-blocker.
\end{lemma}

\begin{proof}
Assume first that $4|n$. For a given $0\leq k\leq \frac{n}{4}-3$, we build a blocker $B_M$ as defined in \cref{Obs:Matrioshka} such that $|V_T|=|V_B|=\frac{n}{4}+k$, $|V_L|=|V_R|=\frac{n}{4}-k$, and $B_T$ and $B_B$ have the structure of $B_Q$ defined in \cref{Obs:Quadrilateral}, for all possible blockers $B_Q$ that were obtained in \cref{Lem:First_Spectrum}.
For a fixed value of $k$, this yields blockers $B_M$ whose sizes cover all the interval $[t_{\min}(k),t_{\max}(k)]$, where\\
\centerline{$t_{\min}(k)=
    \left(\frac{n}{4}+k\right)^2+\left(\frac{n}{4}-k\right)^2+2\left(\frac{n}{4}+k\right)-2=
    \frac{n^2}{8}+\frac{n}{2}+2(k^2+k-1),$}

\medskip  

\centerline{$
     t_{\max}(k)=t_{\min}(k)+2\left(\left(\frac{\left(\frac{n}{4}+k+2\right)^2}{8}+\frac{\frac{n}{4}+k+2}{4}-\frac{11}{8}\right)-\left(\frac{n}{4}+k\right)\right)=
     t_{\min}(k)+\frac{\left(\frac{n}{4}+k\right)^2}{4}-\frac{\frac{n}{4}+k}{2}-\frac{3}{4}.$}
The value $t_{\min}(k)$ is achieved where both $B_T$ and $B_B$ are minimum-sized blockers for $\left(\frac{n}{4}+k+2\right)$-gons, and the value $t_{\max}(k)$ is achieved where these sub-blockers reach the maximal possible size  achieved in \cref{Lem:First_Spectrum}. 
Both $t_{\min}(k)$ and $t_{\max}(k)$ grow for any $k>0$, since $\frac{dt_{\min}(k)}{dk}=2(2k+1)>0$, and $\frac{dt_{\max}(k)}{dk}=\frac{n}{8}+\frac{9k}{2}+\frac{3}{2}>0$. We claim that these intervals overlap for consecutive values of $k$ and a sufficiently large $n$. Indeed, consider
{\scriptsize \begin{eqnarray*}
t_{\max}(k)-t_{\min}(k+1)&=& \left(\left(\frac{n^2}{8}+\frac{n}{2}+2(k^2+k-1)\right)+\frac{\left(\frac{n}{4}+k\right)^2}{4}-\frac{\frac{n}{4}+k}{2}-\frac{3}{4}\right)-\left(\frac{n^2}{8}+\frac{n}{2}+2((k+1)^2+(k+1)-1)\right)\\
&=&\frac{\frac{n}{4}+k}{2}\left(\frac{\frac{n}{4}+k}{2}-1\right)-\frac{19}{4}-4k.
\end{eqnarray*}}
For $k\geq 4$, the last expression is positive for all $n$, and for $0\leq k\leq 3$, it is positive for all $n\geq 25$, as can be directly checked by treating the corresponding quadratic inequalities obtained by setting $k:=0,1,2,3$. Hence, for $n\geq 25$, all the intervals overlap.

The size of the minimal possible blocker built in this way is: $t_{\min}(0)=\frac{n^2}{8}+\frac{n}{2}-2$, and the size of the maximal one is:\\ 
\centerline{
$t_{\max}\left(\frac{n}{4}-3\right)=\left(\frac{n^2}{8}+\frac{n}{2}+2\left(\left(\frac{n}{4}-3\right)^2+\left(\frac{n}{4}-3\right)-1\right)\right)+\frac{\left(\frac{n}{4}+\frac{n}{4}-3\right)^2}{4}-\frac{\frac{n}{4}+\frac{n}{4}-3}{2}-\frac{3}{4}= \frac{n^2}{4}+\frac{n^2}{16}-3n+13.$}

\medskip

If $4  {\not|}  \, n$, we set $|V_T|=\lfloor{\frac{n}{4}}\rfloor+k$, $|V_L|=|V_R|=\lceil{\frac{n}{4}}\rceil-k$, $|V_B|=n-|V_T|-|V_L|-|V_R|$, and a direct calculation shows that for all possible residues, the interval of `covered' values contains the interval $\left[\frac{n^2}{8}+\frac{n}{2}-2,\frac{n^2}{4}+\frac{n^2}{16}-3n+12\frac{13}{16}\right]$.
\end{proof}

\subsubsection{The $(n,t)$-blockers for an integer $t$ in the range $\left[\frac{n^2}{8}+\frac{n}{4}-\frac{11}{8}, \frac{n^2}{8}+\frac{n}{2}-2\right]$
}

There is a gap between the intervals covered in Lemma \ref{Lem:First_Spectrum} and Lemma \ref{Lem:Second_Spectrum}. We cover this interval by the following lemma:

\begin{lemma}\label{Lem:Spectrum_Gap}
 For any $n\geq 21$ and all integers $t$ satisfying $\frac{n^2}{8}+\frac{n}{4}-\frac{11}{8}\leq t\leq  \frac{n^2}{8}+\frac{n}{2}-2$, there exists an $(n,t)$-blocker.
\end{lemma}

\begin{proof}
We prove this lemma by an explicit construction. 

Again, we first assume that $4|n$. Consider a blocker $B_M$ defined as in \cref{Obs:Matrioshka}, where $|V_T|=\frac{n}{2}-1$, $|V_B|=\frac{n}{4}-2$, $|V_L|=3$, $|V_R|=\frac{n}{4}$,  the sub-blocker $B_T$ takes all the values given in \cref{Lem:First_Spectrum} (where we substitute $\frac{n}{2}+1$ for $n$), and $B_B$ takes the minimal possible size for a convex $(|B_B|+2)$-gon, which is $\frac{n}{4}-2$.
The size of the blocker $B_M$ then runs over all the integers in the range from \\
\centerline{$t_{\min}=\left(\frac{n}{2}-1\right)\left(\frac{n}{4}-2\right)+3\cdot \frac{n}{4} +\left(\left(\frac{n}{2}+1\right)-2\right)+\left(\frac{n}{4}-2\right)-2=\frac{n^2}{8}+\frac{n}{4}-3$} 
to\\ 
\centerline{$t_{\max}=\left(\frac{n}{2}-1\right)\left(\frac{n}{4}-2\right)+3\cdot \frac{n}{4} +\left(\frac{\left(\frac{n}{2}+1\right)^2}{8}+\frac{\frac{n}{2}+1}{4}-\frac{11}{8}\right)+\left(\frac{n}{4}-2\right)-2=\frac{n^2}{8}+\frac{n^2}{32}-3.$}

Now, we find the condition on $n$, under which 
$[t_{\min},t_{\max}]$ fills the gap $\left[\frac{n^2}{8}+\frac{n}{4}-\frac{11}{8},  \frac{n^2}{8}+\frac{n}{2}-2\right]$. It is obvious that $t_{\min}<\frac{n^2}{8}+\frac{n}{4}-\frac{11}{8}$. To check the upper bound, we solve the inequality
$t_{\max}>\frac{n^2}{8}+\frac{n}{2}-2$. This inequality is equivalent to the quadratic inequality $\frac{n^2}{32}-\frac{n}{2}-1>0,$
which holds for all $n\geq 18$. Hence, the sizes of the blockers constructed using this strategy cover the interval $\left[\frac{n^2}{8}+\frac{n}{4}-\frac{11}{8},  \frac{n^2}{8}+\frac{n}{2}-2\right]$ for all $n\geq 18$.

\medskip

For the case $4  {\not|}  \, n$, we construct $B_M$ with  $|V_B|=\lfloor\frac{n}{4}\rfloor-2$, $|V_L|=3$, $|V_R|=\lfloor\frac{n}{4}\rfloor$, $|V_T|=n-|V_B|-|V_L|-|V_R|$. A simple calculation shows that in all cases, the covered interval  $[t_{\min},t_{\max}]$ fills the gap $\left[\frac{n^2}{8}+\frac{n}{4}-\frac{11}{8},  \frac{n^2}{8}+\frac{n}{2}-2\right]$ for all $n\geq 21$.
\end{proof}

\subsubsection{The $(n,t)$-blockers for an integer $t$ in the range $\left[\frac{n^2}{4}-2n+10 , \frac{n^2}{2}-C{n\log n}\right]$}

The following lemma completes the proof of \cref{thm:saturation_spectrum} by nesting blockers defined in Lemmas \ref{Lem:First_Spectrum}-\ref{Lem:Spectrum_Gap}, in a fractal manner, inside a specific blocker $B_M$ defined as in \cref{Obs:Matrioshka}. 

\begin{lemma}\label{Lem:Third_Spectrum}
There exist constants $C>0$ and $n_0 \in \mathbb{N}$ such that for all $n>n_0$ and all integers $t$ satisfying $\frac{n^2}{4}-2n+10 \leq t\leq  \frac{n^2}{2}-C{n\log n}$, there exists an $(n,t)$-blocker.
\end{lemma}
\begin{proof}
We first assume that $n$ is even. In order to prove the lemma, we consider the blocker $B_M$ defined in \cref{Obs:Matrioshka} with $|V_L|=|V_R|=3$ and $|V_T|=|V_B|=\frac{n}{2}-3$.
   
To construct the sequence of fractal-type nested blockers, we apply the following algorithm:

\medskip

\begin{itemize}

\item[{\bf{Step 0:}}] We consider the blocker $B_M$ that makes use of the
sub-blockers $B_T$ and $B_B$ (as defined in \cref{Obs:Matrioshka}), where $B_T$ and $B_B$ run over the set of blockers defined in Lemmas \ref{Lem:First_Spectrum}-\ref{Lem:Spectrum_Gap}, each starting from the minimum-sized blocker for an $\left(\frac{n}{2}-1\right)$-gon. We increase alternately the sizes of $B_T$ and $B_B$ by one until both of them reach the maximal size defined in \cref{Lem:Second_Spectrum}. 
    
The minimal possible size of $B_M$ in this step is 
$$t_{\min}=\left(\frac{n}{2}-3\right)\left(\frac{n}{2}-3\right)+3 \cdot 3 +\left(\frac{n}{2}-3\right)+\left(\frac{n}{2}-3\right)-2=\frac{n^2}{4}-2n+10,$$  obtained by taking both $B_T$ and $B_B$ to be minimum-sized blockers. For $n\geq 13$, this value is smaller than the maximal value $\frac{n^2}{4}+\frac{n^2}{16}-3n+{12\frac{13}{16}}$ reached in \cref{Lem:Second_Spectrum}. Hence, for all $n \geq 13$, there will be an overlap between the interval covered by \cref{Lem:Second_Spectrum} and the interval covered by this lemma. 
       
The maximal size is: 
{\footnotesize\begin{eqnarray*}
t^0_{\max}=a_0n^2+b_0n+c_0&=&\left(\frac{n}{2}-3\right)\left(\frac{n}{2}-3\right)+3 \cdot 3 +2\left(\frac{\left(\frac{n}{2}-1\right)^2}{4}+\frac{\left(\frac{n}{2}-1\right)^2}{16}-3\left(\frac{n}{2}-1\right)+12\frac{13}{16}\right)-2\\
&=&\frac{13}{32}n^2-\frac{53}{8}n+\frac{193}{4},
\end{eqnarray*}} 
so we have: $a_0=\frac{13}{32}$, $b_0=-\frac{53}{8}$, $c_0=\frac{193}{4}$. 
Therefore, after step 0, the range $\left[n-2,\frac{13}{32} n^2-\frac{53}{8}n+\frac{193}{4}\right]$ is covered, assuming ${\frac{n}{2}-1} \geq 25$ (this assumption is required for using the results of \cref{Lem:Second_Spectrum}). 

\medskip

\item[{\bf{Step 1:}}] 
We proceed similarly, but now instead of taking $B_T$ and $B_B$ as in Lemmas \ref{Lem:First_Spectrum}-\ref{Lem:Spectrum_Gap}, we take them to be the blockers that were constructed in step 0. As in the previous step, we increase their sizes by one in every iteration. The maximal size of $B_M$ in this step is 
\begin{eqnarray*}
t^1_{\max}=a_1n^2+b_1n+c_1 &=& \left(\frac{n}{2}-3\right)\left(\frac{n}{2}-3\right)+3 \cdot 3 + 2\left(a_0\left(\frac{n}{2}-1\right)^2+b_0\left(\frac{n}{2}-1\right)+c_0\right) -2 \\
&=&\frac{n^2}{4}-3n+16+2\left(a_0\left(\frac{n}{2}-1\right)^2+b_0\left(\frac{n}{2}-1\right)+c_0\right).\end{eqnarray*}

Hence, we get that:
$a_1=\frac{1}{4}+\frac{a_0}{2}, \ b_1=-2a_0+b_0-3 \mbox{ and } c_1=2a_0-2b_0+2c_0+16.$ 
It is easy to verify that after steps 0 and 1, the range 
$\left[n-2,a_1n^2+b_1n+c_1\right]$ is covered.

\medskip

\item[{\bf{Step $i$:}}] Proceeding as in previous step, we continue the saturation spectrum up to the size 
$t^i_{\max}=a_{i}n^2+b_{i}n+c_{i}$,  where
\begin{equation}
\label{Eq:Recurrent_Coefficients}
a_i=\frac{1}{4}+\frac{a_{i-1}}{2},\qquad b_i=-2a_{i-1}+b_{i-1}-3,\qquad c_{i}=2a_{i-1}-2b_{i-1}+2c_{i-1}+16. 
\end{equation}

\end{itemize}

Solving the linear recursions \eqref{Eq:Recurrent_Coefficients} for $a_i$, we get: 
\begin{equation}\label{Eq:explicit_ai}
a_i=\frac{1}{2}-\frac{3}{2^{i+5}}.
\end{equation} 
Next, we have from Recursions \eqref{Eq:Recurrent_Coefficients} for $b_i$ that:
$b_i-b_{i-1}=-2a_{i-1}-3$. Summing these equations up, we get:
$$b_i-b_0=-\sum\limits_{j=0}^{i-1}(2a_j+3)=-\sum\limits_{j=0}^{i-1}\left(1-\frac{3}{2^{j+4}}+3\right) = -4i -\frac{3}{2^{i+3}}+ \frac{3}{8},$$ and hence we have: 
\begin{equation}\label{Eq:explicit_bi}
b_i=-4i-\frac{3}{2^{i+3}}-\frac{25}{4}.\end{equation} 
Substituting 
the explicit expressions we obtained for $a_i$ and $b_i$ in the recursion for $c_i$ in Recursions \eqref{Eq:Recurrent_Coefficients}, and solving the linear recursion for $c_i$, we get:
\begin{equation}\label{Eq:explicit_ci}
c_i= {\frac{689}{8}}\cdot2^i-8i-\frac{75}{2}-\frac{3}{2^{i+3}}.
\end{equation}

Note that while $a_i=\frac{1}{2}-\frac{3}{2^{i+5}}$ increases with $i$ and approaches $\frac{1}{2}$, which seems to get us as close as we wish to $\frac{1}{2}n^2$, the negative term $b_i=-4i-\frac{3}{2^{i+3}}-\frac{25}{4}$ decreases with $i$ and bounds the closeness to $\frac{1}{2}n^2$ we can obtain in this way. Specifically, an easy calculation shows that for any $i$, we have $t^i_{\mathrm{max}}=a_i n^2+b_i n+c_i
\leq \frac{1}{2}n^2-\Omega(n \log n)$. 

\medskip

We choose to perform the algorithm for $k=\log n-\log\log n$ steps, for which this closeness is obtained, where we assume for simplicity that $\log n - \log \log n$ is an integer. For the maximal size of blocker that can be obtained after $\log n - \log \log n$ steps, the explicit expressions for $a_i,b_i$ and $c_i$ obtained in Equations \eqref{Eq:explicit_ai}, \eqref{Eq:explicit_bi} and \eqref{Eq:explicit_ci}
above imply that: 

{\footnotesize\begin{equation*}
a_{\log n - \log \log n}=\frac{1}{2}-\frac{3\log n}{4n},\quad b_{\log n - \log \log n}=-4\log n+4\log\log n-\frac{3\log n}{8n}-\frac{25}{4},\quad
c_{\log n - \log \log n}={\frac{689}{8}}\cdot\frac{n}{\log n}+o(n),\end{equation*}}
and therefore:
\begin{equation}\label{eq:asymp}
t^{\log n - \log \log n}_{\max}=a_{\log n - \log \log n} n^2 + b_{\log n - \log \log n} n +c_{\log n - \log \log n}=\frac{n^2}{2}-\frac{19}{4}\cdot n\log n+o(n\log n).
\end{equation}
Hence, there exists a constant $C>0$ and $n_0>0$ such that for any $n>n_0$, the upper bound of the saturation spectrum is $t_{\max}\geq \frac{n^2}{2}-Cn\log n$.

\medskip 

When $n$ is odd, the above analysis is similar, where the sizes of $V_T$ and $V_B$ are slightly changed to $|V_T|=\frac{n-1}{2}-3$ and $|V_B|=\frac{n+1}{2}-3$ respectively. A direct calculation shows that the change
influences only constants in the $o(n\log n)$ term in Equation \eqref{eq:asymp}. \end{proof}

Lemmas \ref{Lem:First_Spectrum}--\ref{Lem:Third_Spectrum} together complete the proof of \cref{thm:saturation_spectrum}.


\section{The complete characterization of the $(n,n-1)$-blockers} \label{Sec:Full_Char}

In this section, we prove \cref{Thm:Main-Intro} which provides a complete characterization of the $(n,n-1)$-blockers. All the saturated blockers for an $n$-gon where $n \leq 6$ (including the $(n,n-1)$-blockers) are presented in \cref{app:small_cases}, and thus, we can assume in the sequel that $n>6$. 

We start with some observations 
in \cref{subsec:obs}. Then, in Section \ref{subsec:(n,n-1)}, which is the heart of Section \ref{Sec:Full_Char}, we characterize the $(n,n-1)$-blockers containing a non-covered vertex of degree 2. Finally, in Section \ref{subsec:other-case}, we show that the characterization of Section \ref{subsec:(n,n-1)} actually exhausts all $(n,n-1)$-blockers, as any $(n,n-1)$-blocker must contain a non-covered vertex of degree 2.

\subsection{Some observations}\label{subsec:obs}
 
We start with several observations, which will be helpful in the complete characterization of $(n,n-1)$-blockers. 

\begin{observation}\label{Obs:DegAtMost2}
If an ear-cover $e=(i-1,i+1)$ does not belong to the $(n,n-1)$-blocker $B$ for a convex $n$-gon $C$, then $\deg_B(i)\leq 2$. 
\end{observation}

\begin{proof}
Assume towards a contradiction that $\deg_B(i)\geq 3$. Then $B \setminus \{i\}$ is of size at most $(n-1)-3=n-4$, which means that it cannot block all triangulations of the convex $(n-1)$-gon $C \setminus \{i\}$ (by Aichholzer et al.~\cite{AichholzerCMFHHHW10}). Let $T$ be a triangulation of $C \setminus \{i\}$ which it misses. Then $T \cup \{(i-1,i+1)\}$ is a triangulation of $C$ missed by $B$, a contradiction. 
\end{proof}

\begin{observation}\label{Obs:Deg2_non-covered}
If an ear-cover $e=(i-1,i+1)$ does not belong to the $(n,n-1)$-blocker $B$ for a convex $n$-gon $C$ and $\deg_B(i)=2$, then $B\setminus\{i\}$ is a minimum-sized blocker for the $(n-1)$-gon $C\setminus\{i\}$.  
\end{observation}

\begin{proof}
The set $B\setminus\{i\}$ contains $(n-1)-2=n-3$ edges, which is the minimum possible size of a blocker for an $(n-1)$-gon, as was shown by Aichholzer et al.~\cite{AichholzerCMFHHHW10}. Let $T'$ be a triangulation of $C \setminus \{i\}$. Then $T=T' \cup \{i-1,i+1\}$ is a triangulation of $C$, and hence is it blocked by $B$. As $B$ does not contain the edge $\{i-1,i+1\}$ and $T$ does not contain edges emanating from $i$, this implies that $B \setminus \{i\}$ blocks $T'$. Hence, $B \setminus \{i\}$ is a minimum-sized blocker for $C \setminus \{i\}$, as asserted.
\end{proof}

\begin{corollary}\label{Cor:CanConstruct}
If an $(n,n-1)$-blocker $B'$ for an $n$-gon $C'$ has at least one vertex $v$ satisfying $\deg_{B'}(v)=2$, which is not covered by an ear-cover, then $B'$ can be obtained from a minimum-sized $(n-1,n-3)$-blocker $B$ for an $(n-1)$-gon $C$ by adding the vertex $v$ and the two edges emanating from it.
\end{corollary}

\begin{proof}
Follows immediately from \cref{Obs:Deg2_non-covered}. 
\end{proof}

\begin{observation}\label{Obs:non-covered}
For $k\geq n-2$, any $(n,k)$-blocker has at least two vertices which are not covered by ear-covers. 
\end{observation}

\begin{proof}
Any saturated blocker of a convex $n$-gon contains at most $n-2$ ear-covers, otherwise it contains the minimum-sized blocker $B=\{(0,2),(1,3),\dots,(n-3,n-1)\}$ as in \cref{Fig:ExBlockers1}(b) (or a rotation of it) as a subgraph, and hence it is not saturated. Therefore, the number of vertices is at least two more than the number of ear-covers.
\end{proof}
 
Observations \ref{Obs:Isolated},  \ref{Obs:DegAtMost2} and \ref{Obs:non-covered}  imply that in an $(n,n-1)$-blocker there exist at least two vertices that are not covered by ear-covers, and the degree of each such non-covered vertex is either $1$ or $2$. Therefore, we first characterize all $(n,n-1)$-blockers which include at least one non-covered vertex of degree $2$ in a constructive way using \cref{Cor:CanConstruct}, and then we complete the characterization of $(n,n-1)$-blockers by considering the case where all the non-covered vertices in an $(n,n-1)$-blocker are of degree $1$.

\subsection{The $(n',n'-1)$-blockers containing a non-covered vertex of degree~$2$}
\label{subsec:(n,n-1)}

Let $n'=n+1$. In this subsection, we characterize all $(n',n'-1)$-blockers $B'$ of a convex $n'$-gon $C'$ which can be obtained from a minimum-sized $(n,n-2)$-blocker $B$ for a convex $n$-gon $C$ by inserting a new vertex $v$ into $C$ and adding to $B$ two edges $e_1$ and $e_2$ emanating from $v$. The new vertex $v$ will indeed be a non-covered vertex in $C'$ of degree $2$. 

We consider in Propositions \ref{Claim:Beam}--\ref{Claim:MiddleBNNeighbours} below all the possible locations of this new vertex between two consecutive vertices of $C$, which gives us all possible appearances of $B'$.

The structure of the minimum-sized blocker $B$ satisfies the characterization shown in  \cref{Thm:MinBlockerIntro}, and we may assume without loss of generality that the boundary net of $B$ starts at vertex $0$ and ends at vertex $m+2$, while all the beams (if any) emanate from the vertices $m+3,\dots,n-1$. In general, we can insert the new vertex $v$ in one of the following four locations: 
\begin{enumerate}
\item[(1)] outside the boundary net; 
\item[(2)] between the two leftmost or the two rightmost vertices of the boundary net (in this case, the last ear-cover in the boundary net of $B$ becomes a beam in $B'$); 
\item[(3)] between the second and third vertices of a boundary net from the left or from the right (in this case, the two first/last ear-covers in $B$ become intersecting but not conflicting beams in $B'$);
\item[(4)] between any other consecutive vertices in the boundary net (which is possible only if the boundary net in $B$ is at least four edges long).
\end{enumerate}

We consider these cases one by one in the following four subsections (Sections \ref{sec:add_vertex_outside}-\ref{sec:add_vertex_inside}). In each case, we construct all possible $(n',n'-1)$-blockers $B'$ that can be constructed in this way, if there are such.

\subsubsection{First case: adding a new vertex outside the boundary net}\label{sec:add_vertex_outside}

Let $n'=n+1$. In the following proposition, we consider the case of adding the new vertex outside the boundary net: 

\begin{proposition}\label{Claim:Beam}
Given an $(n,n-2)$-blocker $B$, it is possible to obtain an $(n',n'-1)$-blocker $B'$ by inserting a vertex $v$ between two consecutive vertices {$i$, $i+1$ from the set $\{m+2,\dots,n-1,0\}$ and adding two edges $e_1$ and $e_2$ emanating from $v$} in {\em exactly} one of following two situations:
\begin{itemize}

\item[(1)] The beam $(i,\ell)\in B$ is intersected by a unique beam $(k,\ell-1)\in B$ for some $k<i$. In this case, the added edges will be  $e_1=(v,k)$ and $e_2=(v,\ell+1)$. 

\item[(2)] The beam $(i+1,\ell)\in B$ is intersected by a unique beam $(k,\ell+1)\in B$ for some $k>i+1$. In this case, the added edges will be $e_1=(v,k)$ and $e_2=(v,\ell-1)$. 
\end{itemize} 

\end{proposition}

Examples of the transformations defined by this proposition are presented in \cref{Fig:Claim_Beam_explanation}. Parts (a) and (b) in the figure correspond to case (1) and case (2) in the formulation of the proposition, respectively, where the added edges are drawn as bold lines. \begin{figure}[h]
    \centering
    \includegraphics[scale=0.3]{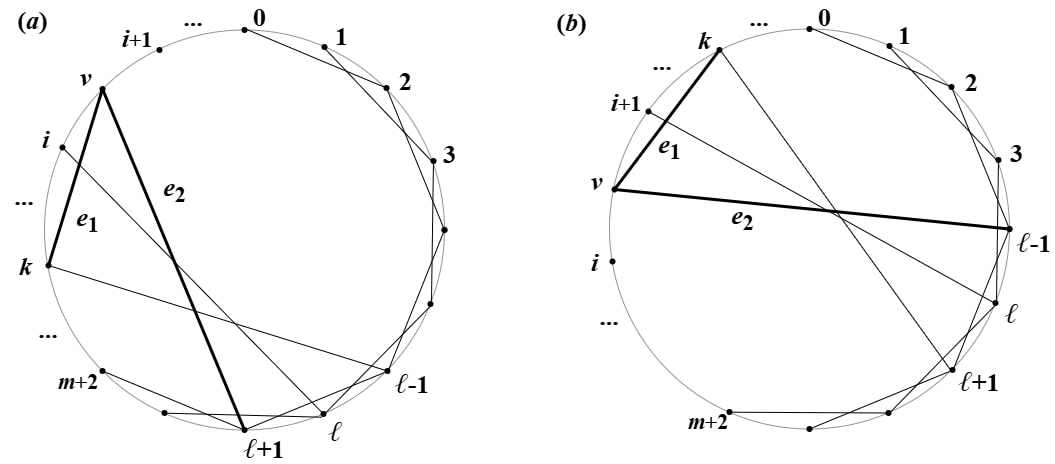}
    \caption{Examples of transforming $B$ to $B'$ in \cref{Claim:Beam} by inserting a vertex $v$ outside the boundary net, and adding two edges emanating from $v$. In both parts, the added edges $e_1$ and $e_2$ are drawn as bold lines.} 
 \label{Fig:Claim_Beam_explanation}
\end{figure}

\begin{proof}
We prove this proposition by constructing all possible $(n',n'-1)$-blockers $B'$ from $B$ by adding a new vertex $v$ outside the boundary net and two edges {$e_1$ and $e_2$} emanating from $v$.
\medskip

For the added edge $e_1$, the following three options are possible for its other endpoint:

\begin{enumerate}
\item[(1)]{ The edge $e_1$ ends outside the boundary net, i.e., $e_1=(v,k)$ for $k>m+2$. By the construction of $B$, there exists a beam $(k,s)\in B$ for some $0<s<m+2$. There are two possibilities for $k$: either $k<i$ or $k>i+1$. We consider the case where $k<i$, and the other case (which leads to the second case in the formulation of Proposition \ref{Claim:Beam}) follows by symmetry.}

{If $k<i$, then we have that ${\rm deg}_{B'}(k)=2$ and $k$ is not covered by an ear-cover (as $k$ is not in the boundary net). 
Consider a blocker $B'\setminus\{k\}$ of the convex $n$-gon $C'\setminus\{k\}$. It is a (minimum-sized) $(n,n-2)$-blocker, where the other added edge $e_2$ can be either an ear-cover or a beam. In both cases, $e_2$ ends in an internal point of the boundary net of $B'$, without conflicting with any other beam, due to the characterization in \cref{Thm:MinBlockerIntro}.}

{We now split our treatment into two sub-cases, depending on the mutual position of $e_2$ and the edge $(k,s)$:}
\begin{itemize}
\item[(1a)]  {If $e_2$ does not conflict with $(k,s)$, then $B'$ is not saturated, since $B'\setminus \{e_1\}$ is also a (minimum-sized) $(n',n'-2)$-blocker.}
     
\item[(1b)]  {If $e_2$ conflicts only with $(k,s)$ (see \cref{Fig:Claim_Beam}(a)), then all the beams emanating from the set $\{k+1,\dots,i\}$ should have the same endpoint $\ell$ in between the end of $e_2$ and $s$ (as otherwise, they would conflict with $e_2$ or with $(k,s)$ in the minimum-sized $(n,n-2)$-blocker $B'\setminus\{k\}$ of the convex $n$-gon $C'\setminus\{k\}$). Moreover, for not causing conflicts apart from that between $e_2$ and $(k,s)$, the distance between their endpoints must be exactly 2, i.e., $(k,s)=(k,\ell-1)$ and $e_2=(v,\ell+1)$.}
\end{itemize}

\begin{figure}[h]
\centerline{\includegraphics[scale=0.3]{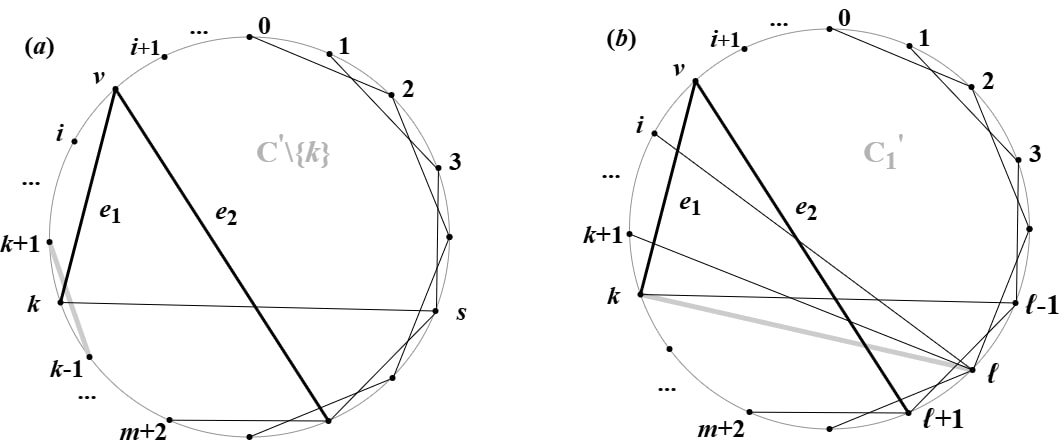}}
\caption{{(a) The construction of $B'$ for Case (1) in the proof of \cref{Claim:Beam}; (b) An illustration for the proof that $B'$ is a blocker for Case (1) in the proof of \cref{Claim:Beam}. In both parts, the added edges $e_1$ and $e_2$ are drawn as bold lines and the edges of the triangulation are drawn in gray.} }
\label{Fig:Claim_Beam}
\end{figure}

\medskip

{Now, we prove that $B'=B\cup\{e_1,e_2\}$ for $e_1=(v,k)$ and $e_2=(v,\ell+1)$ is indeed an $(n',n'-1)$-blocker. First, note that $(k,\ell)\notin B'$ (see \cref{Fig:Claim_Beam}(b)). Any triangulation that does not contain $(k,\ell)$ is blocked by the minimum-sized blocker $B'\cup \{(k,\ell)\}\setminus\{(k,\ell-1),(v,k)\}$, which satisfies the characterization in \cref{Thm:MinBlockerIntro}. Any triangulation that contains $(k,\ell)$ contains also a triangulation of the polygon $C'_1=C'\setminus\{\ell+1,\dots,k-1\}$. That triangulation of $C'_1$ is blocked by a restriction of $B'$ to $C'_1$ which is a minimum-sized blocker, where $(k,\ell-1)$ and $(k+1,\ell)$ turn into ear-covers, and as so, they are a part of the boundary net in the restricted blocker.}
      
{Furthermore, the blocker $B'$ is saturated, since any removal of an edge from $B'$ does not yield a minimum-sized $(n',n'-2)$-blocker, according to the characterization of  \cref{Thm:MinBlockerIntro}.}

\medskip

\item[(2)] The edge $e_1$ ends in an endpoint of the boundary net, i.e., $e_1=(v,0)$ or $e_1=(v,m+2)$. We consider the case  $e_1=(v,0)$ for some $v<n-1$, and the other case will follow by symmetry.

Let $e_1=(v,0)$ for some $v<n-1$. The other added edge, $e_2$, also emanates from $v$, and hence the vertex $0$ is not covered by an ear-cover in $B'$. We consider a triangulation of $C'$ containing the edge $(n-1,1)$ (see \cref{Fig:Claim_Beam_2}). Any triangulation of the $n$-gon $C'\setminus\{0\}$ is blocked by the blocker $B'_1=B'\setminus\{0\}$ which has $n-2$ edges. $B'_1$ is then a (minimum-sized) $(n,n-2)$-blocker which satisfies \cref{Thm:MinBlockerIntro}. Hence, the second endpoint of the last added edge $e_2 \in B_1'$ is one of the internal vertices of the boundary net, with no conflict with any of the existing edges of $B$. In such a case, $B'$ is not saturated, as we remove $(v,0)$ from it without damaging the blocking property, and hence one cannot construct an $(n',n'-1)$-blocker $B'$ when 
the edge $e_1$ ends in an endpoint of the boundary net.

\begin{figure}[h]
\centerline{\includegraphics[scale=0.3]{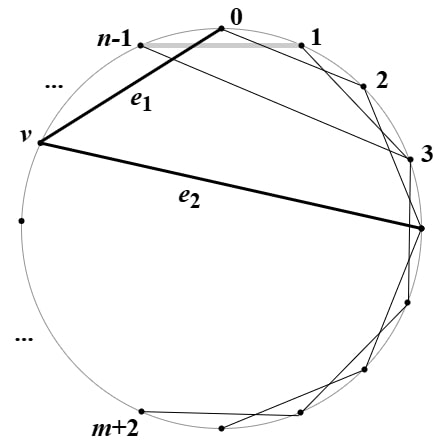}}
\caption{{The construction of $B'$ for Case (2) in the proof of \cref{Claim:Beam}. The added edges $e_1$ and $e_2$ are drawn as bold lines and the edges of the triangulation are drawn in gray.} }
\label{Fig:Claim_Beam_2}
\end{figure}

\medskip

\item[(3)] {The last case is when both edges $e_1$ and $e_2$ end in internal vertices of the boundary net, since if $e_2$ ends outside the boundary net, or in one of its endpoints, we could swap the roles of $e_1$ and $e_2$, and get one of the above cases. We split this case into three sub-cases:}
  
\begin{itemize}
\item[(3a)] {If at least one of the edges $e_1$, $e_2$ does not conflict with any other beam in $B'$, then $B'$ is not saturated, since removing the other added edge from $B'$ would result in a minimum-sized $(n',n'-2)$-blocker as in \cref{Thm:MinBlockerIntro}.}

\item[(3b)] {So we can assume that both edges conflict with some beam in $B'$. We first assume that both edges $e_1$ and $e_2$ conflict with the {\it same} beam $(k,s) \in B'$, where $k>m+2$ (see \cref{Fig:Claim_Beam_3}(a)). Without loss of generality, we can set $v>k$, and then note that $(k,s+1)\notin B'$ (by the characterization of \cref{Thm:MinBlockerIntro}). We claim that in this sub-case, one can construct a triangulation containing $(k,s+1)$ that is not blocked by $B'$. Indeed, in the upper part of $B'$, the vertex $v$ is isolated, and in its lower part, the vertex $k$ is isolated. This allows the construction of a triangulation of $C'$ that misses $B'$, thus showing that $B'$ is not a blocker in this sub-case.}

\begin{figure}[ht]
\begin{center}
\includegraphics[width=0.65\linewidth]{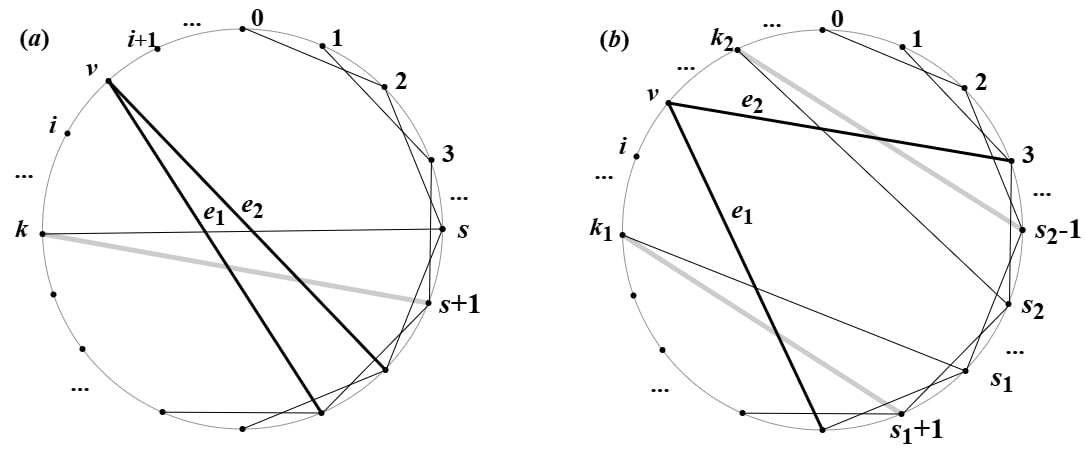}
\end{center}
\caption{{(a) The construction of $B'$ for case (3b) in the proof of \cref{Claim:Beam}, when both added edges conflict with the same beam $(k,s)$; (b) The construction of $B'$ for case (3c) in the proof of \cref{Claim:Beam}, when the added edges conflict with different beams $(k_1,s_1)$ and $(k_2,s_2)$. In both parts, the added edges $e_1$ and $e_2$ are drawn as bold lines and the edges of the triangulation are drawn in gray.} }
\label{Fig:Claim_Beam_3}
\end{figure}

\item[(3c)]{In the last sub-case, we assume that the edge $e_1$ conflicts with a beam $(k_1,s_1)\in B_1'$, and the edge $e_2$ conflicts with a different beam $(k_2,s_2) \in B'$, where $k_1>m+2$ and $k_2>m+2$ (see \cref{Fig:Claim_Beam_3}(b)). Moreover, without loss of generality, we can assume that $k_1<v<k_2$ (i.e., the vertices $k_1$ and $k_2$ are on opposite sides of $v$), as otherwise, one of these beams $(k_j,s_j)$ would conflict with both edges $e_1$ and $e_2$, and we return to Case (3b).  Now, similar to the previous case, note that $(k_1,s_1+1),(k_2,s_2-1)\notin B'$ (by the characterization of \cref{Thm:MinBlockerIntro}), and these two edges do not intersect even if $(k_1,s_1)$ intersects $(k_2,s_2)$ (as these two beams intersect only if $s_1-s_2=1$). Thus, there exists a triangulation containing $(k_1,s_1+1)$ and $(k_2,s_2-1)$ that is missed by $B'$. Indeed,  the middle part of $B'$ contains the isolated vertex $v$, and the lower and upper parts contain the isolated vertices $k_1$ and $k_2$, respectively. Hence, in this sub-case as well, $B'$ is not a blocker.}
\end{itemize}

{Therefore, one cannot construct an $(n',n'-1)$-blocker $B'$ in the case when 
the edges $e_1$ and $e_2$ ends in internal vertices of the boundary net.}
\end{enumerate}

\vspace{-25pt}\end{proof}

\subsubsection{Second case: adding a new vertex between the two last vertices of the boundary net}

Let $n'=n+1$. The following proposition characterizes the $(n',n'-1)$-blockers which can be obtained from a blocker $B$ by inserting the new vertex $v$ of degree $2$ in between the two rightmost or the two leftmost vertices of the boundary net: 
\begin{proposition}\label{Claim:EndNeighbours}
Given an $(n,n-2)$-blocker $B$, whose boundary net starts at vertex $0$ and ends at vertex $m+2$. 
\begin{enumerate}
\item[(a)] It is possible to obtain an $(n',n'-1)$-blocker $B'$ by inserting a new vertex $v$ between the two vertices $0$ and $1$ and adding two edges $e_1$ and $e_2$ emanating from $v$ if and only if $\deg_B(1)=2$. In this case, the added edges will be $e_1=(v,3)$ and $e_2=(v,k)$, where $k>m+2$ and $(k,1)\in B$. 
\item[(b)] It is possible to obtain an $(n',n'-1)$-blocker $B'$ by inserting a new vertex $v$ between the two vertices $m+1$ and $m+2$ and adding two edges $e_1$ and $e_2$ emanating from $v$ if and only if $\deg_B(m+1)=2$. In this case, the added edges will be $e_1=(v,m-1)$ and $e_2=(v,k)$, where $k>m+2$ and $(k,m+1)\in B$.
\end{enumerate}  
\end{proposition}

Examples of the two $(n', n'-1)$-blockers whose construction is described in \cref{Claim:EndNeighbours} are shown in \cref{Fig:Claim_End_Explanation}: part (a) illustrates Case (a) of the proposition, and part (b) illustrates Case (b) of the proposition.

Note that all the beams emanating from vertices between $k$ and $v$ are connected to the same endpoint which is vertex $2$ in Case (a) (or $m$ in Case (b)), as otherwise they would conflict with the beam emanating from $k$ in the initial blocker $B$. 

Below we prove Case (a) of \cref{Claim:EndNeighbours}, and Case (b) follows by symmetry, when replacing the triple of vertices $0,1,3$ by the triple $m+2,m+1,m-1$ respectively. 

\begin{figure}[h]
\centering
\includegraphics[scale=0.33]{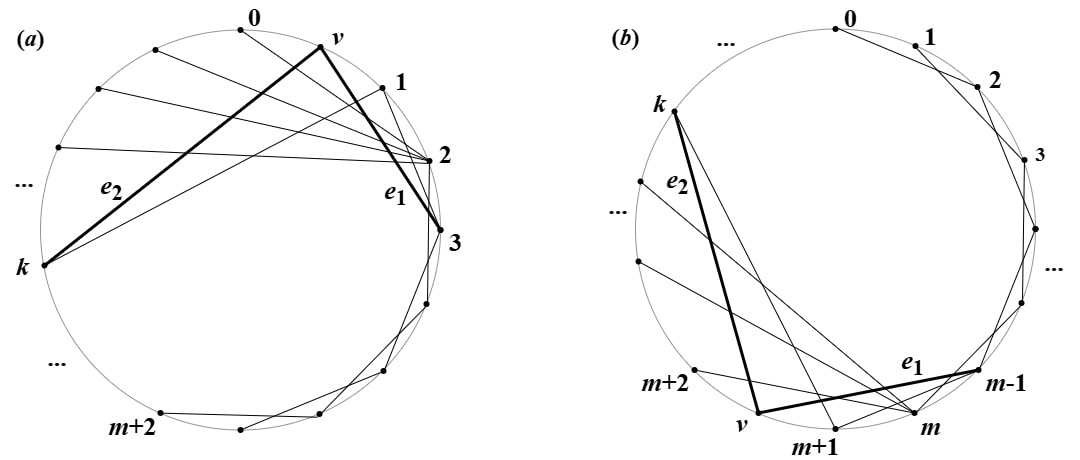}
\caption{Examples of transforming $B$ to $B'$ by inserting a new vertex $v$ in between two end vertices of a boundary net, and adding two edges emanating from $v$, as in Case (a) of \cref{Claim:EndNeighbours} (part (a)) and in Case (b) of \cref{Claim:EndNeighbours}  (part (b)). In both parts, the added edges $e_1$ and $e_2$ are drawn as bold lines.} 
\label{Fig:Claim_End_Explanation}
\end{figure}

\begin{proof}[Proof of Case (a) of \cref{Claim:EndNeighbours}]
We start with the observation that the edge $(v,2)$ cannot be one of the added edges, as otherwise it will be the new last ear-cover of the boundary net in $B'$, and $(0,2)$ is the new beam which does not conflict with any other beam inherited from $B$. Hence, $B'$ is not saturated, as it is possible to remove the other added edge from $B'$.

Moreover, if $\deg_B(1)>2$, then $B'$ is not a blocker, since $B'\setminus\{1\}$ has too few edges ($\leq n-3$) to block all possible triangulations of $C'\setminus\{1\}$ containing $(v,2)\notin B'\setminus\{1\}$.

\medskip

Next, we consider the following two remaining cases for $\deg_B(1)$:

\begin{enumerate}
\item[(1)] If $\deg_B(1)=1$, then there are two sub-cases:
\begin{itemize}
\item[(1a)] If $e_1=(v,3) \in B_1'$, then $e_1$ is a new beam, which does not conflict with any other beam in $B'$ including $(0,2)$, and thus, $B'$ is not saturated, since $B'\setminus \{e_2\}$ for any edge $e_2$ is an $(n',n'-2)$-blocker according to the characterization in   \cref{Thm:MinBlockerIntro}. 
\item[(1b)]
{If $(v,3) \notin B'$, then we consider a triangulation of $C'$ which includes $(v,3)$, and which should be blocked by $B'$.} It follows that either any triangulation of $C'_1=C'\setminus\{1,2\}$ is blocked by $B'_1=B'\setminus\{1,2\}$, or any triangulation of $C'_2=C'\setminus\{0,4,5,\dots,n-1\}$ is blocked by $B'_2=B'\setminus\{0, 4,5,\dots,n-1\}$. In $B'_1$, the vertex $0$ is isolated (since $(v,0)\notin B'$), and $C'_2$ cannot be blocked by $B'_2$ as it is a quadrilateral, and its diagonal $(v,2)$ is not in $B'$ (see \cref{Fig:Claim_End}(a)). Hence, $B_1'$ is not a blocker. 
\end{itemize}

Hence, if  $\deg_B(1)=1$, it is not possible to construct an $(n',n'-1)$-blocker.

\begin{figure}[h]
\centering
\includegraphics[scale=0.33]{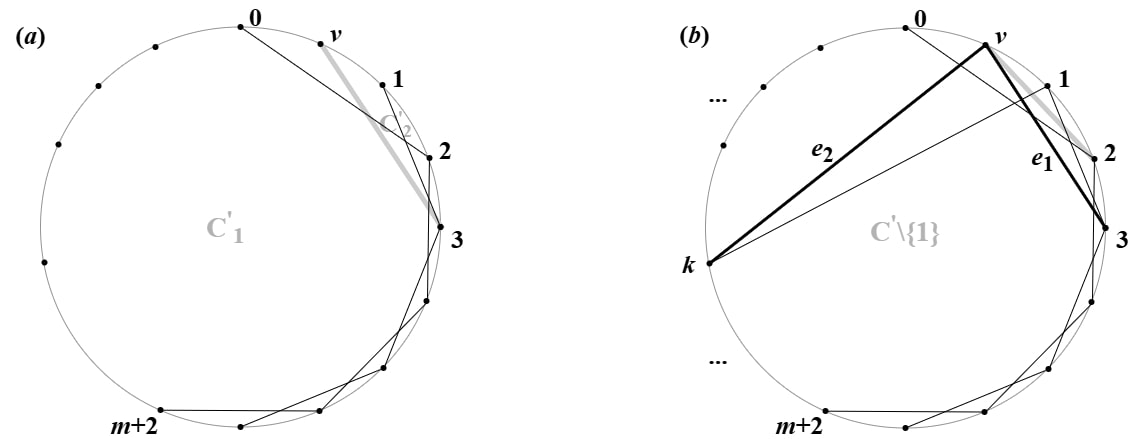}
\caption{Transforming $B$ to $B'$ by inserting a new vertex $v$ in between two end vertices of a boundary net, and adding two edges emanating from $v$ as in \cref{Claim:EndNeighbours}: (a) in the case $\deg_B(1)=1$; (b) in the case $\deg_B(1)=2$. In part (b), the added edges $e_1$ and $e_2$ are drawn as bold lines. In both parts, the edges of the triangulation are drawn in gray.} 
\label{Fig:Claim_End}
\end{figure}

\medskip

\item[(2)] If $\deg_B(1)=2$, then there exists an edge $(k,1)\in B'$. 
By the observation in the beginning of the proof, the choice $e_1=(v,2)$ yields a blocker which is not saturated. 
{Therefore, we can consider a triangulation of $C'$ which contains $(v,2)$ (see \cref{Fig:Claim_End}(b)). This triangulation is blocked by $B'\setminus\{1\}$ which is a (minimum-size) $(n,n-2)$-blocker. The only way to complete such a blocker, which satisfies the characterization of \cref{Thm:MinBlockerIntro}, by adding two edges emanating from $v$, is to add edges $e_1=(v,3)$ and $e_2=(v,k)$. The reason for that is that in the $(n,n-2)$-blocker $B'\setminus\{1\}$, the edge $e_2$ cannot conflict with any other beam that $(1,k)$ was not conflicting with in $B$ and the edge $e_1$ completes the boundary net. }

The resulting $B'$ is indeed a saturated $(n',n'-1)$-blocker: any triangulation which includes $(v,2)$ is blocked by $B'\setminus\{1\}$. Any triangulation which does not include $(v,2)$ is blocked by $B'$, as $B\subset B'$ and $B\cup\{(v,2)\}$ is a minimum-sized $(n', n'-2)$-blocker for $C'$. Also, $B'$ is saturated by its construction. 
\end{enumerate}

\vspace{-12pt}\end{proof}

\subsubsection{Third case: adding a new vertex between the second and the third vertices of the boundary net}

Let $n'=n+1$. In the following proposition, we assume that the boundary net in the initial blocker $B$ has at least $3$ edges, as otherwise, inserting the new vertex in between the second and the third vertex of the boundary net is equivalent to inserting it between the two last vertices of the other side of the boundary net, which is Case (b) of \cref{Claim:EndNeighbours} above: 

\begin{proposition}\label{Claim:NearEndNeighbours}
Given an $(n,n-2)$-blocker $B$, whose boundary net starts at vertex $0$ and ends at vertex $m+2$, where we assume that $m\geq 2$. 
\begin{enumerate}
\item[(a)] It is possible to obtain an  $(n',n'-1)$-blocker $B'$ by inserting a new vertex $v$  between the two vertices $1$ and $2$ and adding two edges $e_1$ and $e_2$ emanating from $v$ if and only if $\deg_B(2)=2$. In this case, the added edges will be $e_1=(v,0)$ and $e_2=(v,4)$.  

\item[(b)] It is possible to obtain an  $(n',n'-1)$-blocker $B'$ by inserting a new vertex $v$  between the two vertices $m$ and $m+1$ and adding two edges $e_1$ and $e_2$ emanating from $v$ if and only if $\deg_B(m)=2$. In this case, the added edges will be $e_1=(v,m-2)$ and $e_2=(v,m+2)$.  
\end{enumerate}

\end{proposition}

Examples of the two $(n', n'-1)$-blockers whose construction is described in \cref{Claim:NearEndNeighbours} are shown in \cref{Fig:Claim_NEnd_Explanation}: part (a) illustrates Case (a) of the proposition and part (b) illustrates Case (b) of the proposition.

\begin{figure}[h]
\centering
\includegraphics[scale=0.3]{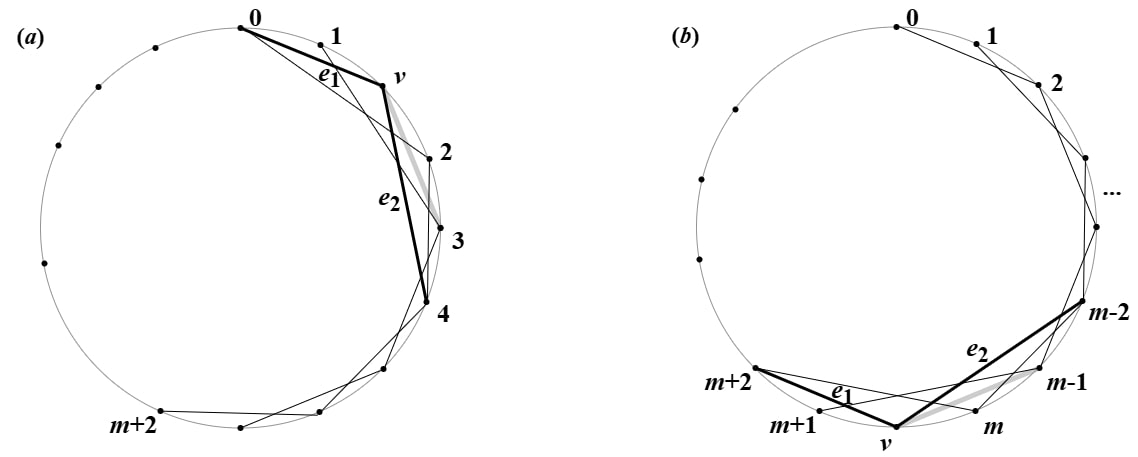}
\caption{Examples of transforming $B$ to $B'$ by inserting a new vertex $v$ in between vertices: (a) $1$ and $2$; (b) $m$ and $m+1$ of the boundary net, and adding two edges emanating from $v$, as in \cref{Claim:NearEndNeighbours}. In both parts, the added edges $e_1$ and $e_2$ are drawn as bold lines and the edge of the triangulation is drawn in gray.} 
\label{Fig:Claim_NEnd_Explanation}
\end{figure}

Note that there are no beams connected to vertex $1$ in Case (a) (resp., $m+1$ in Case (b)), otherwise they conflict with the beam emanating from $n-1$ in Case (a) (resp., $m+3$ in Case (b)) in the initial blocker $B$.

Below we prove Case (a) of \cref{Claim:NearEndNeighbours}, and Case (b) follows by symmetry, when replacing  the vertices $(0,1,2,3,4)$ by $(m+2,m+1,m,m-1,m-2)$. 

\begin{proof}[Proof of Case (a) of \cref{Claim:NearEndNeighbours}]
We start with the observation that the edge $(v,3)$ cannot be one of the added edges, as otherwise $B'$ is not saturated, since $B'\setminus\{e_2\}$ satisfies \cref{Thm:MinBlockerIntro} for any edge $e_2$. So we consider a triangulation of $C'$ which contains $(v,3)$, and we conclude the following (see \cref{Fig:Claim_NEnd_Explanation}): 
\begin{itemize}
\item $\deg_B(2)=2$: $\deg_B(2)\geq 2$, as the vertex $2$ is an internal vertex of the boundary net and hence two ear-covers emanate from $2$. Also,  $\deg_B(2)\leq 2$, as otherwise $B'\setminus\{2\}$ has less than $n-2$ edges and cannot be a blocker to $C'\setminus\{2\}$. Therefore, $B'\setminus\{2\}$ is a minimum-sized $(n,n-2)$-blocker.

\item It is necessary that $e_1=(v,0)$, as otherwise the vertex $0$ would be isolated in $B'\setminus\{2\}$.

\item It is necessary that $e_2=(v,4)$: the vertex $4$ satisfies exactly one of the following two conditions: 
\begin{enumerate}
\item[(1)] The vertex $4$ is an internal vertex of the boundary net of the blocker $B$, and then according to \cref{Thm:MinBlockerIntro}, the minimum-sized $(n,n-2)$-blocker $B'\setminus\{2\}$ should contain the ear-cover $e_2=(v,4)$.
\item[(2)] The vertex $4$ is an endpoint of the boundary net of the blocker $B$, and therefore the vertex $4$ would be an  isolated vertex in the $(n,n-2)$-blocker $B'\setminus\{2\}$, unless it is, following \cref{Thm:MinBlockerIntro}, connected to vertex $v$ or to the vertex $1$. As the added edge $e_2$ should emanate from $v$, the only possibility is: $e_2=(v,4)$.    
\end{enumerate} 
\end{itemize}
     
Therefore, the added edges in this case must be $e_1=(v,0)$ and $e_2=(v,4)$. 

Now, we have to show that $B'=B\cup\{(v,0),(v,4)\}$ is indeed a saturated $(n',n'-1)$-blocker. 
Indeed, any triangulation which includes $(v,3)$ is blocked by $B'\setminus\{2\}$. Any triangulation which does not include $(v,3)$ is blocked by $B\cup\{(v,3)\}$, which is a minimum-sized $(n',n'-2)$-blocker for $C'$. Also, $B'$ is saturated by its construction. 
\end{proof}

\subsubsection{Fourth case: adding a new vertex in the middle of the boundary net}
\label{sec:add_vertex_inside}

Let $n'=n+1$. The following proposition completes all possible ways to construct an $(n',n'-1)$-blocker with a non-covered vertex of degree $2$. Here, we consider the case of inserting the new vertex $v$ in the middle of the boundary net of an $(n,n-2)$-blocker $B$, such that there are at least $3$ vertices of the boundary net from each side of $v$. This case is only possible if the boundary net of the original blocker $B$ consists of at least $4$ ear-covers (i.e. $m\geq 3$). 

\begin{proposition}\label{Claim:MiddleBNNeighbours}
Given an $(n,n-2)$-blocker $B$, whose boundary net starts at vertex $0$ and ends at vertex $m+2$, where we assume that $m\geq 3$. 

It is possible to obtain an $(n',n'-1)$-blocker $B'$ by inserting a new vertex $v$ in between two consecutive vertices $i, i+1$ from the set $\{2,3,\dots,m\}$, and adding two edges $e_1$ and $e_2$ emanating from $v$, in exactly one of the following three ways (see \cref{Fig:Claim_Middle_Explanation}):
\begin{enumerate}
\item[(1)] If the added edges are $e_1=(v,i-1)$ and $e_2=(v,i+2)$.
\item[(2)] If $\deg_B(i+1)=2$, then the added edges are $e_1=(v,i-1)$ and $e_2=(v,i+3)$.
\item[(3)] If $\deg_B(i)=2$, then the added edges are $e_1=(v,i-2)$ and $e_2=(v,i+2)$.
\end{enumerate}
\end{proposition}

\medskip

\begin{figure}[h]
\centering
\includegraphics[width=\linewidth]{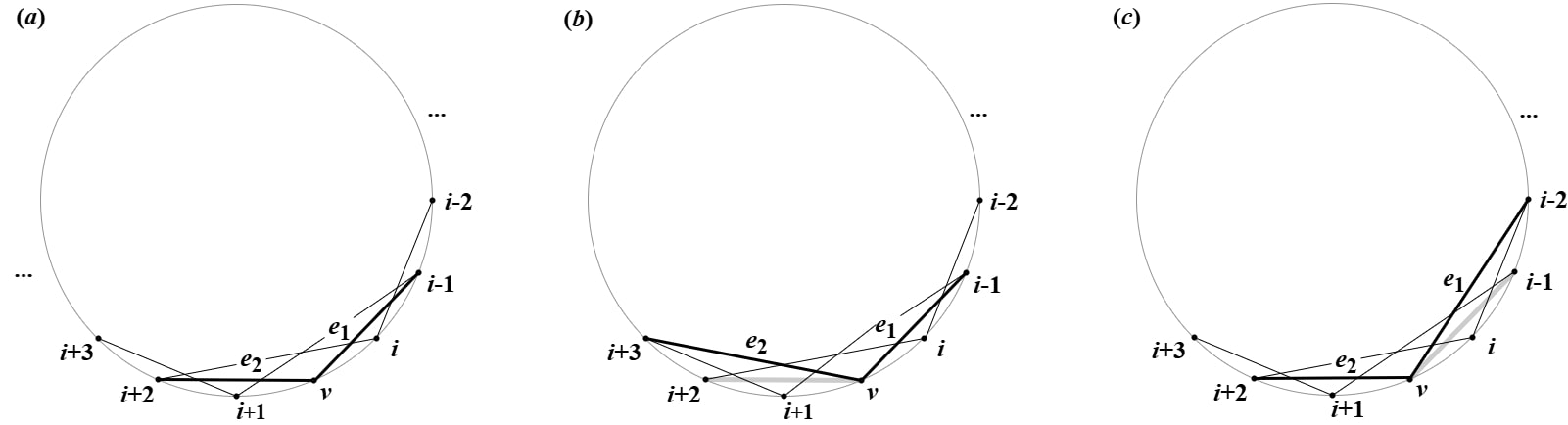}
\caption{Possible transformations of $B$ to $B'$ by inserting a new vertex $v$ in between internal vertices of the boundary net and adding two edges $e_1$ and $e_2$ emanating from $v$ as in \cref{Claim:MiddleBNNeighbours}.  The edges $\{(i-2,i)$, $\dots$, $(i+1,i+3)\}$ are part of the boundary net in $B$, including the cases where $i=2$ or $i=m-1$. In all parts, the added edges $e_1$ and $e_2$ are drawn as bold lines, and the edge of the triangulation is drawn by gray.} 
\label{Fig:Claim_Middle_Explanation}
\end{figure}

\medskip

\begin{proof}
We start by noting that $\deg_B(i)\geq 2$ and $\deg_B(i+1)\geq 2$ (as the vertices $i$ and $i+1$ are in the middle of the boundary net, and hence two ear-covers emanate from each of them). Hence, at least one of added edges $e_1$ or $e_2$ must be an ear-cover, as otherwise, it is not possible to block all triangulations of $C'$ which include $(v,i-1)$ and $(v,i+2)$ by $B'$, since $B'\setminus\{i,i+1\}$ contains at most $n'-5$ edges, which are too few edges to block all triangulations of the $(n'-2)$-gon $C'\setminus\{i,i+1\}$. 

Therefore, we fix $e_1=(v,i-1)$ to be an ear-cover (where the other case $e_1=(v,i+2)$ follows by symmetry). We consider three possible cases to choose the endpoint of the edge $e_2$ that emanates from $v$:
\begin{enumerate}
\item[(1)] $e_2=(v,i+2)$ (see \cref{Fig:Claim_Middle_Explanation}(a)). Then $B'=B \cup \{e_1,e_2\}$ is an $(n',n'-1)$-blocker. Indeed, any triangulation which includes the diagonal $(i,i+1)$ is blocked by the original blocker $B$, and any triangulation which does not include this diagonal is blocked by $B'$, since $B'\cup\{(i,i+1)\}$ contains the minimum-sized $(n',n'-2)$-blocker $B'\cup\{(i,i+1)\}\setminus\{(i,i+2),(i-1,i+1)\}$, which satisfies the characterization of \cref{Thm:MinBlockerIntro}. 

Also, $B'$ is saturated, since any removal of an edge from $B'$ gives a graph which does not satisfy the characterization of \cref{Thm:MinBlockerIntro}. Hence, Case (1) of the proposition is proven.

\medskip

\item[(2)] $e_2\neq(v,i+2)$ (see \cref{Fig:Claim_Middle_Explanation}(b)). This case is possible only if $\deg_B(i+1)=2$, otherwise $i+1$ is a non-covered vertex satisfying $\deg_B(i+1)>2$, which is impossible in an $(n',n'-1)$-blocker by Observation \ref{Obs:DegAtMost2}.

So we assume that $e_2\neq(v,i+2)$ and $\deg_B(i+1)=2$. Then, we consider a triangulation of $C'$ which includes the diagonal $(v,i+2)$. This triangulation 
contains a triangulation of $C'\setminus\{i+1\}$ and is blocked by $B'\setminus\{i+1\}$ which is a minimum-sized $(n,n-2)$-blocker. This blocker should satisfy the characterization of \cref{Thm:MinBlockerIntro}, which is possible only if $e_2=(v,i+3)$. Indeed, the vertex $i+3$ belongs to the boundary net in $B$, and is either the endpoint of the boundary net (and then it is isolated in  $B'\setminus\{i+1\}$, unless it is connected to $v$), or its middle point (and then it should be a middle point of the boundary net also in $B'\setminus\{i+1\}$ due to the characterization of \cref{Thm:MinBlockerIntro}). 

We claim that the set of edges $B'=B\cup\{(v,i-1),(v,i+3)\}$ is a saturated $(n',n'-1)$-blocker. Indeed, any triangulation which includes the diagonal $(i,i+1)$ is blocked by the original blocker $B$, any triangulation which includes the diagonal $(v,i+2)$ is blocked by $B'$ as it contains $B'
\setminus\{i+1\}$ which is a blocker by construction, and any triangulation which does not include both diagonals $(i,i+1)$ and $(v,i+2)$ is blocked by $B'$, since $B'\cup\{(i,i+1),(v,i+2)\}$ contains the minimum-sized $(n',n'-2)$-blocker $B'\cup\{(i,i+1), (v,i+2)\}\setminus\{(i-1,i+1),(i,i+2),(v,i+3)\}$, which satisfies the characterization of  \cref{Thm:MinBlockerIntro}. 

Furthermore, $B'$ is saturated, since any removal of an edge from $B'$ gives a set of edges which does not satisfy the characterization of \cref{Thm:MinBlockerIntro}. Hence, Case (2) of the proposition is also proven.
\end{enumerate}

Case (3) of the proposition follows by symmetry from Case (2), when replacing the
vertices $(0,1,2,\dots,$ $n-1)$ by $(m+2,m+1,m,\dots,m+3)$ (see \cref{Fig:Claim_Middle_Explanation}(c)).
\end{proof}

\subsubsection{Complete  characterization of $(n,n-1)$-blockers having at least one non-covered vertex of degree $2$ using the subgraphs ``seagull'', ``butterfly'' and ``bouquet''}

Now, we are ready to characterize all  $(n,n-1)$-blockers, which contain at least one non-covered vertex of degree $2$. By Corollary \ref{Cor:CanConstruct},
we have that any such blocker $B'$ of the $n$-gon $C'$ can be obtained from some minimum-sized $(n-1,n-3)$-blocker $B$ by adding a new vertex $v$ to it and two edges emanating from this new vertex $v$.
Propositions \ref{Claim:Beam}--\ref{Claim:MiddleBNNeighbours} give us three possible subgraphs, called ``seagull'', ``butterfly'' and ``bouquet'', defined in the formulation of Theorem \ref{Thm:Main-Intro} (see Figure~\ref{Fig:Bouquet1}), which appear in (saturated) $(n,n-1)$-blockers containing a non-covered vertex of degree $2$ and obtaining from a minimum-sized blocker by adding a vertex $v$ to it and two edges emanating from it, and do not appear in minimum-sized $(n,n-2)$-blockers:

\medskip

\begin{remark}\label{obs:3types}
    Any $(n,n-1)$-blocker, constructed using Propositions \ref{Claim:Beam}--\ref{Claim:MiddleBNNeighbours} above, 
    includes either a single ``bouquet'' subgraph, a single ``butterfly'' subgraph, or a single ``seagull'' subgraph:   

\begin{itemize}
\item \cref{Claim:Beam} provides a ``bouquet'' subgraph (see Figure~\ref{Fig:Bouquet1}(c)) with an arbitrary positive number of beams, and a base which can be any ear-cover in the boundary net of the initial blocker.
Note that if there is more than one beam in a ``bouquet'' subgraph, then all of the beams start from a single vertex $\ell$ covered by a base. This is inevitable, as otherwise, by the construction from the initial minimum-sized blocker, the added beam (here, $(k,\ell-1)$ or $(k+t,\ell+1)$) would conflict with more than one other beam. 

\item \cref{Claim:EndNeighbours} (resp., \cref{Claim:NearEndNeighbours}) provides a ``bouquet'' subgraph with a base (i.e., the edge $(\ell-1,\ell+1)$) being the last ear-cover in the boundary net, and with an arbitrary number, greater than $1$, of beams (resp., exactly $2$ beams).

\item A ``seagull'' subgraph (see Figure~\ref{Fig:Bouquet1}(a)) appears only in Case (1) of \cref{Claim:MiddleBNNeighbours}. In this case, both vertices $\ell-2$ and $\ell+2$ are covered by ear-covers in $B'$. 

\item Cases (2) and (3) of  \cref{Claim:MiddleBNNeighbours} give a ``butterfly'' subgraph (see Figure~\ref{Fig:Bouquet1}(b)), where at least one of the vertices $\ell-2$ and $\ell+3$, or both, are covered by an ear-cover in $B'$. Also, if the vertices $\ell-2$ or $\ell+3$ is not covered by an ear-cover, then this ``butterfly'' subgraph is equivalent to a ``bouquet'' subgraph with a single beam $(\ell-1,\ell+2)$.
\end{itemize}

\end{remark}

\begin{remark}\label{remark butterfly bouquet}
Note that a ``butterfly'' subgraph can be also considered as a ``bouquet'' subgraph with the single beam inside its ``vase'' (i.e., inside the subgraph $\{(\ell-1,\ell+1),(k,\ell-1),(k+t,\ell+1), (k,k+t)\}$). To avoid ambiguities, we classify an $(n,n-1)$-blocker with a ``butterfly'' subgraph as a blocker of Type 3, if one of the end vertices of the ``butterfly'' (i.e., either $\ell-2$ or $\ell+3$) is not covered by an ear-cover.
\end{remark}

Based on Corollary  \ref{Cor:CanConstruct}, we can now supply a complete characterization for $(n,n-1)$-blockers which contain at least one non-covered vertex of degree $2$: 
\begin{theorem}\label{Thm:Result_deg2}
Any $(n,n-1)$-blocker which contains at least one non-covered vertex of degree $2$ is of one of the three types appearing in the formulation of \cref{Thm:Main-Intro}.

\end{theorem}

\begin{proof}
We first verify the theorem for $n\leq 6$: 
\begin{itemize}
\item For $n=4$, the only possible $(4,2)$-blocker contains both diagonals, and there are no $(4,3)$-blockers.
\item For $n=5$, any edge is an ear-cover, and there are no $(5,4)$-blockers, see \cref{app:small_cases}.
\item For $n=6$ (a case which was analyzed in detail in Appendix~\ref{app:small_cases}), the only $(n,n-1)$-blocker, shown in \cref{Fig:Thm2_Mincases}(a) below,  is of Type 3.
\end{itemize}  
Hence, in the sequel, we can assume that $n\geq 7$. Let $B'$ be an $(n,n-1)$-blocker of the $n$-gon $C'$, which contains an uncovered vertex $v$ satisfying $\deg_{B'}(v)=2$. By  Corollary \ref{Cor:CanConstruct}, there exists a minimum-sized $(n-1,n-3)$-blocker  $B=B'\setminus\{v\}$, satisfying that $B'$  can be obtained from $B$ by adding to it a new vertex $v$ and two edges emanating from this new vertex $v$.
All the possible ways to insert into convex $(n-1)$-gon $C$ a new vertex $v$ and adding two edges $e_1$ and $e_2$ to an $(n-1,n-3)$-blocker $B$ emanating from this new vertex $v$ in order to obtain an $(n,n-1)$-blocker $B'$ of a convex $n$-gon $C'$ are already described in Propositions \ref{Claim:Beam}--\ref{Claim:MiddleBNNeighbours}. 
Before considering the resulting $(n,n-1)$-blockers by the transformations presented in these propositions, we verify that all these propositions are indeed applicable to construct an $(n,n-1)$-blocker for all $n \geq 7$:
\begin{itemize}
\item  For applying \cref{Claim:Beam}, the initial $(n-1,n-3)$-blocker $B$ should have at least two ear-covers and two intersecting beams, and thus, it should have at least six vertices (i.e., $n-1 \geq 6$). 
\item For applying \cref{Claim:EndNeighbours}, the blocker $B$ should again have at least two ear-covers and two beams, and hence, once again it should have at least six vertices  (i.e., $n-1 \geq 6$). 
\item For \cref{Claim:NearEndNeighbours}, the blocker $B$ should have at least three ear-covers, and hence $n-1\geq 5$. 
\item For \cref{Claim:MiddleBNNeighbours}, the blocker $B$ should have at least four ear-covers, and hence, $n-1\geq 6$. 
\end{itemize}

\medskip

Next, we consider the resulting $(n,n-1)$-blockers obtained by all the possible transformations that can be applied for all $n\geq 7$ (see Remarks \ref{obs:3types} and \ref{remark butterfly bouquet} above):
\begin{itemize}
\item For $n\geq 7$, the application of \cref{Claim:Beam} to an $(n-1,n-3)$-blocker $B$ yields an arbitrary $(n,n-1)$-blocker of Type 3, with any non-zero number of beams inside and outside the ``bouquet'', and an  arbitrary $B_1$. 

The minimal case, $n=7$, is shown in \cref{Fig:Thm2_Mincases}(b). In this case, the initial $(6,4)$-blocker consists of two ear-covers and two intersecting beams (see Figure \ref{Fig:Ex6}(c) below).

\item For \cref{Claim:EndNeighbours}, the initial blocker $B$ has at least two ear-covers and two beams ($(n-1,2)$ and $(k,1)$). For $n\geq 7$, the transformation gives an $(n,n-1)$-blocker of Type 3 with an arbitrary number of beams inside and outside the ``bouquet'', and either $B_{\rm 1,L}=B_{\rm 2,L}=\emptyset$ or $B_{\rm 1,R}=B_{\rm 2,R}=\emptyset$. 

The minimal case, $n=7$, is shown in \cref{Fig:Thm2_Mincases}(c). In this case, the blocker is of Type 3, where $B_1$ and $B_2$ are empty sets. 

\item Any transformation according to \cref{Claim:NearEndNeighbours} for $n\geq 7$, gives an $(n,n-1)$-blocker of Type 3 with the ``bouquet'' having only one beam inside, and either $B_{\rm 2,L}$ or $B_{\rm 2,R}$ is empty. 

The minimal case, $n=7$, is presented in \cref{Fig:Thm2_Mincases}(d). In this case, it is a blocker of Type 3 where $B_1$ and $B_2$ are empty sets. 

\item \cref{Claim:MiddleBNNeighbours} gives the following outcomes: Case (1) results in an arbitrary $(n,n-1)$-blocker of Type 1, while Cases (2) and (3) result in either a blocker of Type 2, or a blocker of Type 3, where either $B_{\rm 1,L}=B_{\rm 2,L}=\emptyset$ or $B_{\rm 1,R}=B_{\rm 2,R}=\emptyset$, and  $\Bouq$ contains only one beam (see Remark \ref{remark butterfly bouquet} above). 
 
The minimal case for Case (1), $n =7$, is shown in \cref{Fig:Thm2_Mincases}(e). The minimal case for Cases (2) and (3) yields a blocker similar to \cref{Fig:Thm2_Mincases}(b), where the added vertex is $v=4$.  
\end{itemize}

\vspace{-19pt}\end{proof}

\begin{figure}[H]
\centering
\includegraphics[width=0.9\linewidth]{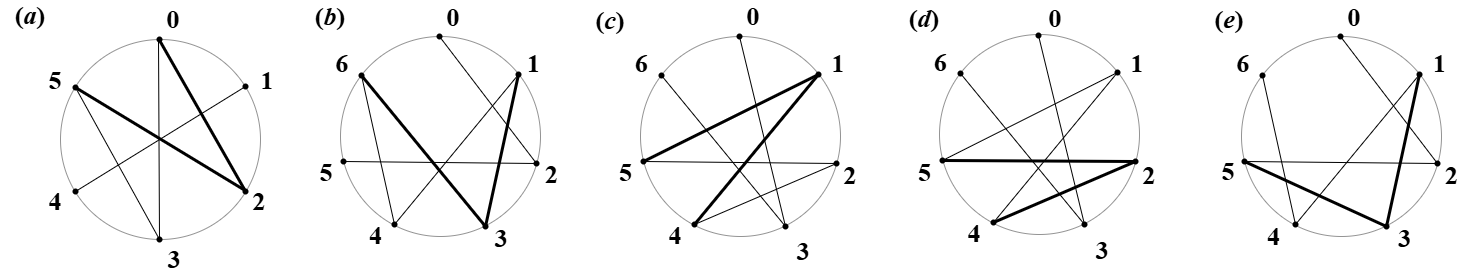}
\caption{The minimal cases of $(n,n-1)$-blockers constructed by the transformations in Propositions~\ref{Claim:Beam} -- \ref{Claim:MiddleBNNeighbours}: (a) for $n=6$; (b) for \cref{Claim:Beam}, $n=7$; (c) for \cref{Claim:EndNeighbours}, $n=7$; (d) for \cref{Claim:NearEndNeighbours}, $n=7$; (e) for Case (1) of \cref{Claim:MiddleBNNeighbours}, $n=7$. In all parts, the added edges are drawn as bold lines.} \label{Fig:Thm2_Mincases}
\end{figure}

\medskip

We conclude this subsection with the following observation, which will be used in the next subsection:

\begin{observation}\label{Obs:Types_2}
    Out of all $(n,n-1)$-blockers appearing in \cref{Thm:Main-Intro}, only the blocker of Type 1 (which contains a ``seagull'') has a single non-covered vertex of degree $2$. The $(n,n-1)$-blockers of Type 2 and Type 3 have at least two such vertices each.
\end{observation}

\subsection{The complete characterization of all $(n,n-1)$-blockers}
\label{subsec:other-case}

In this subsection, we complete the characterization of all $(n,n-1)$-blockers, by showing that each $(n,n-1)$-blocker must contain a non-covered vertex of degree $2$, which implies that the characterization given in Section~\ref{subsec:(n,n-1)} actually exhausts all $(n,n-1)$-blockers.

\medskip

For this, we start by proving the following technical lemma, which shows that if there exists an $(n,n-1)$-blocker with no non-covered vertices of degree $2$, then it would admit a very specific structure:
\begin{lemma}\label{Lem:NotMin}
Let $B'$ be an $(n,n-1)$-blocker having no non-covered vertices of degree $2$. If $v$ is a vertex satisfying $\deg_{B'}(v)=1$ and $(v-1,v+1)\notin B'$, then $B=B'\setminus\{v\}$ is an $(n-1,n-2)$-blocker.
\end{lemma}

\begin{proof}
By the construction, $B$ is a blocker of size $n-2$ for the $(n-1)$-gon $C=C'\setminus\{v\}$, as otherwise the triangulation $T\cup \{(v-1,v+1)\}$ of $C'$ is missed by the blocker $B'$ where $T$ is the triangulation of $C$ missed by $B$.  

For proving that $B$ is indeed a (saturated) $(n-1,n-2)$-blocker, we assume towards a contradiction that $B$ is not saturated. It means that  there exists an edge $e_1=(i,j) \in B$, where $i,j\neq v$, such that $B^*=B\setminus\{e_1\}$ is a blocker for the convex $(n-1)$-gon $C$ having $n-3$ edges and thus it is a minimum-sized blocker and hence satisfying the characterization of \cref{Thm:MinBlockerIntro}. Therefore, it is possible to construct $B'$ for the convex $n$-gon $C'=C \cup \{v\}$ from a minimum-sized $(n-1,n-3)$-blocker $B^*$ by inserting the vertex $v$, and two edges $e_1=(i,j)$ and $e_2=(v,k)$, where $i,j\neq v$. 
    
Since $\deg_{B'}(v)=1$, then according to \cref{Obs:Covered}, there should exist an ear-cover $(k-1,k+1) \in B'$, and hence $(k-1,k+1) \in B=B' \setminus \{v\}$.

\medskip

We divide our treatment into the following six cases, according to the different possibilities to insert the new vertex $v$ between the vertices of $B^*$ and the edge $e_2$ emanating from $v$, and we will show that none of them is possible:
\begin{enumerate}
\item[(1)] The vertex $v$ is located anywhere, and $e_2=(v,k)$ ends in one of the vertices of $B^*$ which are not covered by an ear-cover (i.e. at the end of a beam or at the last vertex of the boundary net). Then by \cref{Obs:Covered}, it is necessary that $e_1=(i,j)=(k-1,k+1)$, and thus either the vertex $k-1$ or $k+1$ in $B'$ has degree $2$ and is not covered by an ear-cover, which contradicts the condition on  $B'$ in the formulation of the lemma, and hence is impossible.
        
\medskip
        
\item[(2)] The vertex $v$ is located outside the boundary net of $B^*$ and the edge  $e_2=(v,k)$ ends in one of the internal points of the boundary net of $B^*$, without conflicts with other beams. Then $B^*\cup\{e_2\} \subsetneq B'$ is a minimum-sized $(n,n-2)$-blocker for $C'$ and so $B'$ is not saturated, which again contradicts the conditions of $B'$ in the lemma.
        
\medskip
        
\item[(3)] The vertex $v$ is located outside the boundary net of $B^*$ and the edge $e_2=(v,k)$ ends in one of the internal points of the boundary net of $B^*$ conflicting with a beam $(\ell,s)$. Without loss of generality, we assume that $\ell<v$ and $s<k-1$ (so $(v,k)$ and $(\ell,s)$ indeed conflict), see \cref{Fig:Lem_Min}(a). 
        
Now, considering the triangulation of $C'$ which includes $(v,k-1)\notin B'$, we conclude that $B'$ should block either any triangulation of $C'_1=C'\setminus\{v+1,\dots,0,\dots,k-2\}$, or of $C'_2=C'\setminus\{k,\dots,v-1\}$. In $B'_2$, which is the restriction of $B'$ to $C'_2$, the vertex $v$ is an isolated vertex, as $e_1$ cannot emanate from $v$. Hence, $e_1$ should appear in $C'_1$ and emanate from $\ell$, as otherwise the vertex $\ell$ is isolated in $B'_1=B'\setminus\{v+1,\dots,0,\dots,k-2\}$.
However, $\ell$ is the endpoint of the beam in $B^*$, hence $\deg_{B'}(\ell)=2$ and $\ell$ is not covered by an ear-cover. This contradicts the conditions of $B'$ in the lemma.

\medskip

\begin{figure}[H]
\centering
\includegraphics[width=\linewidth]{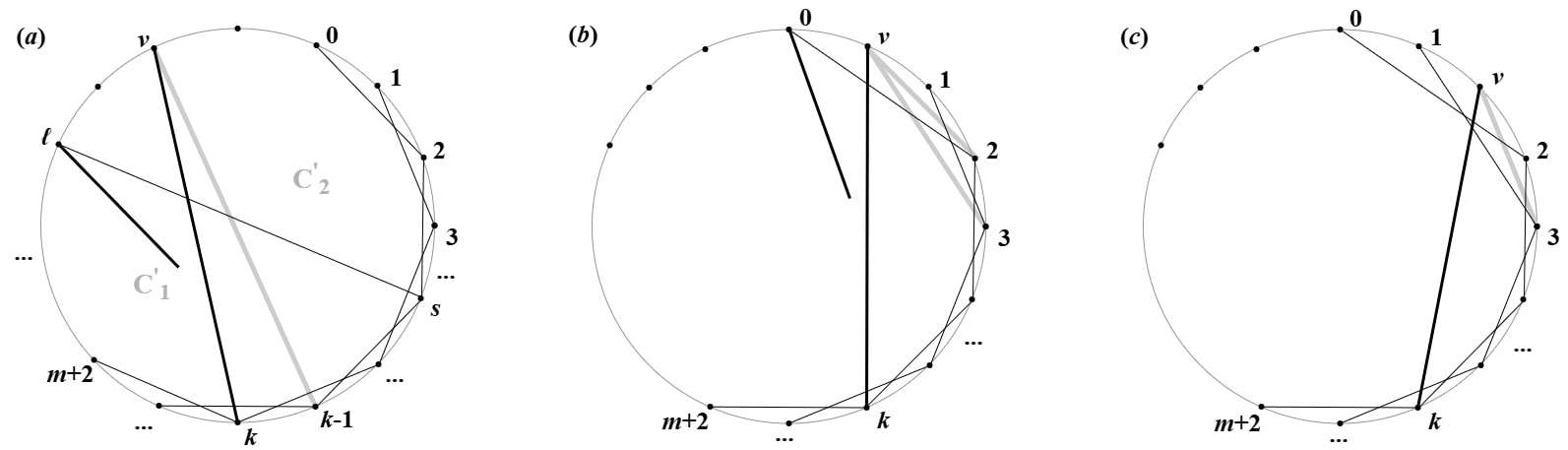}
\caption{Constructing an $(n,n-1)$-blocker $B'$ in \cref{Lem:NotMin}: Part (a) illustrates Case (3), part (b) illustrates Case (4) and part (c) illustrates Case (5). In all parts, the edge $(v,k)$ is drawn as a bold line and edges of the triangulation are drawn in gray.} \label{Fig:Lem_Min}
\end{figure}

\medskip
        
\item[(4)] The vertex $v$ is inserted between the two last vertices of the boundary net of $B^*$, and the edge $e_2$ emanates from $v$ to another internal point $k$ of the boundary net (note that the case that $k$ is not an internal point of the boundary net, was already dealt in Case (1) above). Without loss of generality, we place the vertex $v$ between the vertices $0$ and $1$, see \cref{Fig:Lem_Min}(b).

Note that $e_2\neq(v,2)$, as otherwise $B^*\cup\{e_2\}$ is already a minimum-sized $(n-1,n-3)$-blocker, and then $B'$ is not saturated for any $e_1$.  Also $\deg_{B'}(1)=1$, as otherwise $B'$ has a non-covered vertex of degree $2$, which contradicts the conditions of $B'$ in the lemma. Therefore, in $B^*$, there are no beams emanating from the vertex $1$.
Moreover, we have that $e_2\neq(v,3)$, as otherwise $B^*\cup\{e_2\}$ is already a minimum-sized $(n-1,n-3)$-blocker as in \cref{Thm:MinBlockerIntro}, and hence $B'$ again is not saturated. 

Hence, we have $e_2=(v,k)$ for some $k\in\{4,\dots,m+1\}$.
We then consider a triangulation of $C$ which contains $(v,3)$ to determine the conditions regarding the other added edge $e_1$. To block this triangulation, $e_1$ must emanate either from $v$ -- which is impossible as $\deg_{B'}(v)=1$, or from the vertex $0$, and then in $B'$ the vertex $0$ is a non-covered vertex of degree $2$, which again contradicts the conditions of $B'$ in the lemma.
        
\medskip
        
\item[(5)] The vertex $v$ is inserted in between two near-end vertices of the boundary net of $B^*$ (which has at least three segments), and the edge $e_2$ emanates from $v$ to another internal point of the boundary net (as in the previous case, the case that $k$ is not an internal point of the boundary net, was already dealt in Case (1) above). Without loss of generality, we place the vertex $v$ between vertices $1$ and $2$ (see \cref{Fig:Lem_Min}(c)).

Note that $\deg_{B^*}(1)=1$, as otherwise in $B'$, the vertex 1 is a non-covered vertex of degree 2 or more ($e_2=(v,0)$ could be an ear-cover for it, which is impossible in this case), which contradicts the condition of the lemma.
In this case, $e_2\neq(v,3)$, since otherwise $B^*\cup\{e_2\}$ is already a minimum-sized blocker, and $B'$ is not saturated for any $e_1$.  Also, $\deg_{B^*}(2)\geq 2$, since in $B^*$, there are two ear-covers: $(0,2)$ and $(2,4)$. Therefore, the vertex $2$ in $B'$ is a non-covered vertex satisfying $\deg_{B'}(2) \geq 2$, which contradicts the conditions of $B'$ in the lemma.

\medskip
        
\item[(6)] The vertex $v$ is inserted in between two consecutive internal vertices $i$ and $i+1$ of the boundary net of $B^*$ from the set $\{2,\dots,m\}$, where $m\geq 3$.

Note that $\deg_{B^*}(i)\geq 2$ and $\deg_{B^*}(i+1)\geq 2$ as they are internal vertices of the boundary net. Wherever $e_2$ emanates from $v$, one of vertices $i$ or $i+1$ is a non-covered vertex in $B'$ of degree at least $2$, which contradicts the conditions of $B'$ in the lemma.
\end{enumerate}

\medskip

As none of the above cases is possible, we have that $B=B'\setminus\{v\}$ is indeed a saturated $(n-1,n-2)$-blocker, as asserted.
\end{proof}

\medskip

\noindent \cref{Lem:NotMin} allows us to complete the proof of \cref{Thm:Main-Intro}: 
\begin{proof}[Proof of \cref{Thm:Main-Intro}]
By \cref{Thm:Result_deg2}, for completing the proof of \cref{Thm:Main-Intro}, it is enough to prove that any $(n,n-1)$-blocker has a non-covered vertex of degree $2$.

Assume towards a contradiction that there exists an $(n,n-1)$-blocker $B_0$ with no non-covered vertices of degree $2$, and let the vertex $v^{(0)}$ be a non-covered vertex of maximal degree, i.e. $\deg_{B_0}\hspace{-3pt}\left(v^{(0)}\right)=1$. By \cref{Lem:NotMin}, one can construct a finite sequence of blockers $B_1=B_0\setminus\left\{v^{(0)}\right\}$, $B_2=B_1\setminus\left\{v^{(1)}\right\}$, and so on,  up to $B_p=B_{p-1}\setminus\left\{v^{(p-1)}\right\}$, where $B_p$ has a non-covered vertex of degree $2$. As the sizes of blockers in this sequence decrease by $1$ at each step, and the minimum case $(6,5)$-blocker $B_{(6,5)}$  has non-covered vertices of degree $2$ (see \cref{Ex:6} in \cref{app:small_cases}), such a process will always stop, and therefore a blocker $B_p$ exists for any initial blocker $B_0$.

Now, we consider the last two blockers in this sequence: the $(n-p+1,n-p)$-blocker $B_{p-1}$ and the $(n-p,n-p-1)$-blocker $B_p$. The blocker $B_{p-1}$ has no non-covered vertices of degree $2$, and, after removing the vertex $v^{(p-1)}$ from it, the resulting set $B_p$ has non-covered vertices of degree $2$. It implies that, first, $B_p$ has a single vertex $v^{(p)}$ satisfying $\deg_{B_p}\hspace{-3pt}\left(v^{(p)}\right)=2$, and, second, the edge emanating from $v^{(p-1)}$ in $B_{p-1}$ (which was deleted from $B_{p-1}$ in order to obtain $B_p$) is an ear-cover which covers $v^{(p)}$ (as otherwise $v^{(p)}$ was an non-covered vertex of degree at least $2$ in $B_{p-1}$, a contradiction), see \cref{Fig:Thm_Main_Intro}.

\medskip

\begin{figure}[h]
\centering
\includegraphics[scale=0.4]{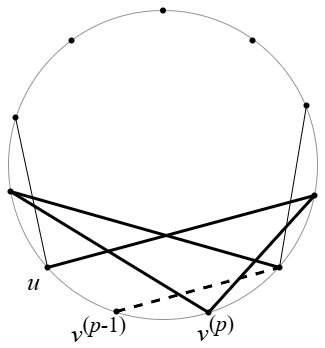}
\caption{Transforming $B_{p-1}$ to $B_{p}$ in the proof of \cref{Thm:Main-Intro}. The ``seagull'' subgraph is drawn with bold lines, two necessary ear-covers in the blocker of Type 1 are drawn with thin lines (as in this case, the new vertex $v$ is added as an internal vertex of the boundary net), and the removed edge is drawn as a dashed line. } 
\label{Fig:Thm_Main_Intro}
\end{figure}

\medskip

By \cref{Obs:Types_2}, the only subgraph having a {\it single} non-covered vertex of degree $2$ is a blocker which contains a ``seagull'' subgraph (a blocker of Type 1, as described in 
\cref{Thm:Result_deg2}), and the blocker $B_{p-1}$ can be obtained from the blocker $B_p$ of Type 1 by inserting a vertex $v^{(p-1)}$ to the left or to the right of $v^{(p)}$, and drawing the ear-cover which covers $v^{(p)}$ from $v^{(p-1)}$. 
Without loss of generality, assume that  $v^{(p-1)}$ is located to the left of $v^{(p)}$. Then the vertex $u$, which is the vertex to the left of $v^{(p-1)}$, is also a non-covered vertex in $B_{p-1}$ of degree at least $2$ (see \cref{Fig:Thm_Main_Intro} and the explanation in its caption), which contradicts the assumption that $B_p$ is the first blocker in the sequence which has non-covered vertices of degree more than $1$.

By this contradiction, it follows that any $(n,n-1)$-blocker has at least one non-covered vertex of degree $2$, and hence, by Theorem \ref{Thm:Result_deg2}, all $(n,n-1)$-blockers satisfy the characterization of \cref{Thm:Main-Intro}.
\end{proof}

\appendix
\section{The saturated blockers for $n \leq 6$ vertices}\label{app:small_cases}

In this appendix, we characterize all saturated blockers of convex polygons with $n \leq 6$ vertices.
\begin{itemize}
\item For $n=3$, the convex polygon (i.e., a triangle) has no diagonals, and so there are no triangulations and no blockers.

\item For $n=4$, the convex polygon (i.e., a quadrilateral) has only two diagonals, and the unique possible blocker is a $(4,2)$-blocker, which consists of both of them; obviously, it is saturated (see Figure~\ref{Fig:Min}(a)). 

\item For $n=5$, the polygon (i.e., a pentagon) has $5$ diagonals. The only $(5,3)$-blocker (up to rotation) is shown in Figure~\ref{Fig:Min}(b). The only possible blocker with $4$ edges (up to rotation) is given in Figure~\ref{Fig:Min}(c), and is obviously not saturated, as it includes a $(5,3)$-blocker. Similarly, there are no saturated blockers with $5$ edges, and hence, the only saturated blockers are rotations of the $(5,3)$-blocker depicted in Figure~\ref{Fig:Min}(c).

\begin{figure}[h]
\centering
\includegraphics[scale=0.4]{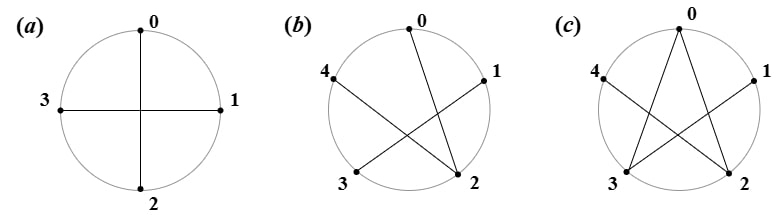}
\caption{(a) A $(4,2)$-blocker; (b) a $(5,3)$-blocker; (c) a blocker with $4$ diagonals for a pentagon (which is not saturated).} 
\label{Fig:Min}
\end{figure}

\end{itemize} 

The minimal case for which there exists a saturated blocker which is not minimum-sized is $n=6$. In \cref{Fig:Ex6}(a-c) we present all possible $(6,4)$-blockers up to rotation, according to the characterization proved in~\cite{KellerS20}. Any (saturated) $(6,5)$-blocker should not include one of them as a subgraph.
\begin{figure}[h]
\centering
\includegraphics[width=0.85\linewidth]{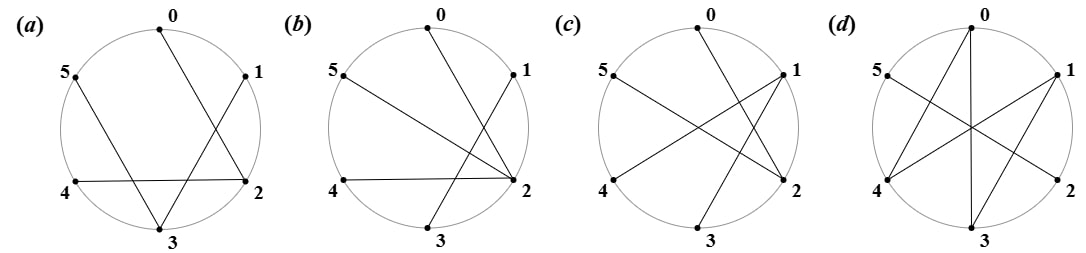}
\caption{(a)-(c) Various $(6,4)$-blockers; (d) a $(6,5)$-blocker.} 
\label{Fig:Ex6}
\end{figure}

\begin{example}[The minimal case for a saturated blocker which is not minimum-sized]\label{Ex:6}
The set of edges $B_{(6,5)}=\{(0,3),(0,4),(1,3),(1,4),(2,5)\}$ is a $(6,5)$-blocker for a convex $6$-gon (see \cref{Fig:Ex6}(d)). 
\end{example}

\begin{proof}
A convex $6$-gon has $9$ diagonals, 5 of them are occupied by $B_{(6,5)}$. The four remaining diagonals $(2,0)$, $(2,4)$, $(5,1)$ and $(5,3)$ consist of two pairs of intersecting diagonals, and hence, any triangulation of the convex 6-gon can contain only two of them, which is not enough, as $3$ diagonals are required. Thus, $B_{(6,5)}$ is indeed a blocker.

It is also saturated, since no edge we remove from $B_{(6,5)}$ would lead to one of the three $(6,4)$-blockers appearing in \cref{Fig:Ex6}(a-c) (up to rotation).
\end{proof}

Note that the saturated $(6,5)$-blocker $B_{(6,5)}$ can be obtained from the minimum-sized $(5,3)$-blocker (see Figure \ref{Fig:Min}(b)) by adding a vertex and two edges emanating from it (similar to the procedure in Proposition \ref{Claim:NearEndNeighbours}, as in the $(5,3)$-blocker, all the edges are parts of the boundary net), see also in the proof of Theorem \ref{Thm:Result_deg2} and Figure \ref{Fig:Thm2_Mincases}(a) above. 

\medskip

An exhaustive search shows that $B_{(6,5)}$ given in \cref{Ex:6} is the only $(6,5)$-blocker (up to rotation), and that there are no saturated blockers for a convex $6$-gon with a larger number of edges.

\bibliographystyle{abbrv}
\bibliography{bibliography}

\end{document}

%% file: macros.tex
\usepackage{graphicx} 
\usepackage[utf8]{inputenc}
\usepackage{mathtools}
\usepackage{xcolor} 
\usepackage{amsmath}
\usepackage{amssymb} 
\usepackage{amsthm}
\usepackage{enumitem}

\usepackage{cleveref}
\renewcommand{\cref}{\Cref}

\usepackage[left=2cm,right=2cm, top=2cm,bottom=2cm,bindingoffset=0cm]{geometry}

\newcommand{\TriF}{{\mathcal{T}}}

\newcommand{\ex}{{\mathrm{ex}}}
\newcommand{\sat}{{\mathrm{sat}}}
\newcommand{\Sea}{B_{\mathrm{Sea}}}
\newcommand{\Buf}{B_{\mathrm{Buf}}}
\newcommand{\Bouq}{B_{\mathrm{Bouq}}}

\newtheorem{theorem}{Theorem}[section]

\newtheorem{observation}[theorem]{Observation}
\newtheorem{corollary}[theorem]{Corollary}
\newtheorem{lemma}[theorem]{Lemma}

\newtheorem{example}[theorem]{Example}
\newtheorem{proposition}[theorem]{Proposition}

\usepackage[font=footnotesize,labelfont=bf]{caption}